\documentclass[journal]{IEEEtran}
\usepackage{graphicx}          
\usepackage{amsmath}
\usepackage{amsthm}
\usepackage{amssymb}
\usepackage{color}

\usepackage{savesym} \savesymbol{AND}
\usepackage{romannum}

\usepackage{algorithmic}
\usepackage{algorithm}

\newtheorem{theorem}{Theorem}

\newtheorem{lemma}[theorem]{Lemma}

\newtheorem{remark}[theorem]{Remark}

\newcommand{\X}{\mathcal{X}}
\newcommand{\D}{\mathcal{D}}
\newcommand{\I}{\mathcal{I}}
\newcommand{\A}{\mathcal{A}}
\newcommand{\R}{\mathbb{R}}
\newcommand{\N}{\mathbb{N}}

\ifCLASSINFOpdf
\else
\fi
\begin{document}
%
\title{
Boundary Feedback Control of $2\times2$ Quasilinear Hyperbolic Systems: Predictive Synthesis and Robustness Analysis
}
%
%
%

\author{Timm~Strecker,
       Ole~Morten~Aamo
        and~ Michael~Cantoni
\thanks{Timm Strecker and Michael Cantoni are with the Department of Electrical and Electronic Engineering, The University of Melbourne, Australia (timm.strecker@unimelb.edu.au; cantoni@unimelb.edu.au). Ole Morten Aamo is with the Department of Engineering Cybernetics, Norwegian University of Science and Technology (NTNU), Norway (aamo@ntnu.no).}
\thanks{This work was supported by the Australian Research
Council (LP160100666).}}

\maketitle

\begin{abstract}
We present a  predictive feedback control method for a class of quasilinear hyperbolic systems with one boundary control input.   Assuming exact model knowledge, convergence to the origin, or tracking at the uncontrolled boundary, are achieved in finite time. A robustness certificate is provided, showing that at least under more restrictive assumptions on the system coefficients, the control method has inherent robustness properties with respect to small errors in the model, measurements and control input.    Rigorous, although conservative conditions on the time derivative of the initial condition and on the design parameter controlling the convergence speed are given to ensure global existence of the solution for initial conditions with arbitrary infinity-norm.
\end{abstract}


%
\IEEEpeerreviewmaketitle

\section{Introduction}
We consider 1-d quasilinear hyperbolic systems of the form
%
%
%
\begin{align}
w_t(x,t) = \Lambda(x,w(x,t))\,w_x(x,t) + F(x,w(x,t)), \label{w_t}
\end{align}
%
%
%
where subscripts denote partial derivatives,  $x\in[0,1]$, $t\geq 0$ and
\begin{align}
w(x,t)&=\left(\begin{matrix} u(x,t)&v(x,t) \end{matrix} \right)^T,
\end{align}
with boundary conditions
 \begin{align}
u(0,t) &= g^u(v(0,t),t) \label{uBC}\\
v(1,t) & = U(t) + g^v(u(1,t),t)\label{vBC}
\end{align}
and initial condition
\begin{align}
w(\cdot,0)=w_0=\left( \begin{matrix}u_0&v_0 \end{matrix} \right)^T.  \label{uvIC}
\end{align}
Without loss of generality, we assume 
\begin{equation}
g^v(u(1,t),t)\equiv 0,
\end{equation}
otherwise we could adjust the control input as $U(t) = \tilde{U}(t)-g^v(u(1,t),t)$.
 The nonlinear functions $\Lambda$ and $F$ are of the form
\begin{align}
\Lambda(x,z) &=\left( \begin{matrix}-\lambda^u(x,z) & 0\\ 0& \lambda^v(x,z)    \end{matrix}  \right), & \lambda^u,\,\lambda^v&>0, \\
F(x,z) &=\left( \begin{matrix} f^u(x,z)& f^v(x,z)    \end{matrix}\right)^T,
\end{align}
where  $\lambda^u$ and $\lambda^v$ are assumed to be positive. The functions $f^u$, $f^v$, $g^u$ and $g^v$ are nonlinear functions of the state and space and $U$ is the control input to be designed.  Precise assumptions on the system coefficients are given in Section \ref{sec:assumptions}. 

Equations (\ref{w_t})-(\ref{uvIC}) are widely used to model 1-d transport phenomena such as gas or fluid flow through pipelines, open-channel flow, traffic flow, electrical transmission lines, and blood flow in arteries \cite{bastin2016book}. In many cases, sensors and actuators are restricted to the boundary of the spatial domain.

 We consider the problem of stabilizing  (\ref{w_t})-(\ref{uvIC}) at the origin or, alternatively, to satisfy a tracking objective at the uncontrolled boundary, using a sampled-time control law.

\subsection{Background}

The boundary control of systems of form (\ref{w_t})-(\ref{uvIC}) has been considered in several papers. In \cite{greenberg1984effect}, dissipative boundary conditions  are designed   that asymptotically stabilize the system. Similar static boundary feedback laws are developed in \cite{coron2007strict} by use of a control Lyapunov function. However, such static feeback laws work only for small in-domain coupling terms (represented by $F$ in (\ref{w_t})) \cite{bastin2016book} and do not  actively steer the system to  the origin in finite time. The exact, finite time boundary controllability and observability of 1-dimensional hyperbolic systems   is well established for both linear \cite{russell1978controllability} and quasilinear \cite{cirina1969boundary,li2003exact,li2008observability} systems.  A predictor-feedback control law for a  PDE-ODE cascade is presented in \cite{bekiaris2018compensation}. However, there is only one PDE-state and no source terms, which renders it equivalent to a system with input-dependent
input delay.   

The proof of controllability in \cite{li2003exact} is constructive,~i.e.,~it provides a method of computing the inputs that steer the system to the origin. However, it is an open-loop control law where the inputs are pre-computed for a long time interval  and then applied in an open-loop fashion. There is no direct way to update the control inputs based on measurements of the state as time proceeds. This makes performance and stability properties of the controlled system sensitive to disturbances and model  uncertainty.  Further, in case the origin is not reached exactly at the end of the control interval due to errors, even existence of a solution is no longer guaranteed. A similar method for tracking of a reference signal is presented in \cite{gugat2011flow}.

Feedback controllers  (utilizing either state or output measurements) that steer systems of form (\ref{w_t})-(\ref{uvIC}) to the origin in the theoretical minimum time are developed  in \cite{vazquez2011backstepping} for linear systems using the backstepping method,  and in \cite{strecker2017output} for semilinear systems using a predictive control method. These results have also been generalized to various broader system classes and bilateral boundary control, see e.g.  \cite{strecker2017twosided, strecker2017seriesinterconnections,strecker2019semilineargeneralheterodirectional} for the semilinear case and  \cite{aamo2013disturbance,dimeglio2013stabilization,hu2015control,
auriol2016two,deutscher2017finite} for linear backstepping controllers.  It is also shown in \cite{coron2013local} that quasilinear systems can be stabilized locally by the  linear backstepping controller that is obtained by linearizing the quasilinear system and applying the design from \cite{vazquez2011backstepping}. The controllability result \cite{li2003exact}  is also local, but only to ensure existence of the solution, not because  a linearization is applied.   There exist quasilinear, controllable systems for which the linear design does not work  (see Section \ref{sec:numerical example} for an example). 

\subsection{Contributions}

In this paper, we extend the methodology presented in \cite{strecker2017output}  from semilinear to  quasilinear systems.  The idea is to predict the states in the interior of the domain up to the time when they are reached by the control input, and then to solve the dynamics backwards along the characteristic lines to obtain the control inputs that are needed to reach pre-designed target states. The details are developed in Section \ref{sec: exact state feedback}, where the approach is illustrated in Figure \ref{fig:outline control}.

 The  challenge in quasilinear systems is that the transport speed $\Lambda$ and, thus, the characteristic lines of the system depend on the state and the control input.  This complicates the computation of the control inputs. Moreover, even if the states remain bounded, the solution can cease to exist when the gradient, which is governed by a PDE with quadratic right-hand side, escapes in finite time \cite{bressan2000hyperbolic}.  Therefore, besides steering the system to the desired state, the control law must also be designed to prevent a collision of  characteristic lines by keeping the gradients bounded.    The need to keep control of the gradients restricts the choice of the state space. We consider a  type of weak, Lipschitz-continuous  solutions called the broad solution  \cite{bressan2000hyperbolic}. While our design can be modified to classical $C^1$-solutions, as considered in \cite{cirina1969boundary,li2003exact},  discontinuous $L^{\infty}$-solutions as in the linear \cite{vazquez2011backstepping} and semilinear \cite{strecker2017output} cases are not possible in quasilinear systems. 

We establish rigorous, although conservative, bounds on the growth of both the state and its derivatives. These form the basis for proving existence and predictability of the trajectories until the time they are affected by the control input, and for establishing a sufficient condition for global (in time) existence of the closed-loop solution. By making a clear distinction between the evolution of the state and its  derivative, we can allow initial conditions with arbitrary infinity-norm and only need to assume small time-derivatives, instead of requiring small $C^1$-norm as in \cite{li2003exact,cirina1969boundary}. This idea was first explored in \cite{gugat2003global}, but only for a particular conservation law ($F\equiv 0$) which simplified the derivations significantly. 

Our condition for global (in time) existence of the closed-loop trajectory allows us to prove that at least under more restrictive assumptions, the proposed feedback control law has inherent  robustness properties. This result, which to our knowledge is the first of its kind for this type of system, certifies that for small uncertainties, the trajectories converge to an arbitrarily small ball around the origin. Importantly, the solution does not cease to exist in finite time. This is a significant advancement compared to the results in \cite{cirina1969boundary,li2003exact}, which are based on semi-global solutions on bounded time-horizons (i.e., the solution exists up to a finite, given time; see e.g.~\cite{cirina1970semiglobalsolutions,li2001semiglobal}). These results rely on the solution reaching an equilibrium in finite time and then staying there, but cannot exclude blow-up of the solution  if the equilibrium is not reached exactly because of modelling errors.

This paper builds on the preliminary work reported in \cite{strecker2019quasilinearfirstorder}, which is restricted to a system with only one state and lacks the rigorous bounds on the allowable time derivatives and the robustness analysis.

\subsection{Outline}
 In Section \ref{sec: preliminaries}, we present technical preliminaries, including bounds on the growth of  the state and its derivatives and sufficient conditions for global existence of the solution. In Section \ref{sec: exact state feedback}, we present our control law that steers the system to the origin in finite time, assuming that the model parameters are  known exactly.  An alternative tracking problem at the uncontrolled boundary is considered in Section \ref{sec:tracking}. 
Our robustness result is shown in Section \ref{sec:robustness}.  Section \ref{sec conclusion} gives some concluding remarks  and several technical proofs are given in the appendix.

\section{Preliminaries}\label{sec: preliminaries}
\subsection{System assumptions and notation}\label{sec:assumptions}
For a Lipschitz continuous function $h:\,\A\rightarrow \R^n$,   $\A\subset\R^m$, $m,n\in \{1,2\}$, denote the minimal Lipschitz constant by
\begin{equation}
l(h) =\sup_{z_1\in\A,\,z_2\in\A} \frac{\|h(z_1)-h(z_2)\|_{\infty}}{\|z_1-z_2\|_{\infty}}.
\end{equation}
The set of  Lipschitz-continuous functions on $\A$ with infinity-norm bounded by some $c\in\R$ and minimal Lipschitz constant bounded by $c^{\prime}\in\R$ is denoted by
\begin{equation}
\begin{aligned}
\X_{\A}^{c,c^{\prime}}= &\left\{h:\,\mathcal{A}\rightarrow \R^n ~|~ h \text{ is Lipschitz continuous},\right.\\
& ~\left.\sup_{z\in\A} \|h(z)\|_{\infty} \leq c ,\, l(h) \leq c^{\prime} \right\}.
\end{aligned}
\end{equation}
Note that by Rademacher's theorem (see e.g.~\cite[Theorem 2.8]{bressan2000hyperbolic}), Lipschitz-continuous functions are differentiable almost everywhere. 
Following \cite[Chapter 3]{bressan2000hyperbolic}, we consider broad solutions of (\ref{w_t})-(\ref{uvIC}), which are functions in $\X_{[0,1]\times [0,T]}^{c,c^{\prime}}$ for $T\in\R\cup\{\infty\}$ and some appropriate $c,\,c^{\prime}>0$.
\begin{remark}
The broad solution of (\ref{w_t})-(\ref{uvIC}) is the solution of the integral equations that are obtained by integrating  along the characteristic lines. Although the differential equations (\ref{w_t}) might  not be satisfied in the classical sense, the solution satisfies the PDEs almost everywhere. Throughout the paper we readily switch between differential and integral form.
\end{remark}
%

The initial conditions $w_0$ and the control input $U$ are assumed to be Lipschitz  and satisfy the compatibility conditions
 \begin{align}
 u_0(0)&=g^u(v_0(0),0) \\
 U(t) &= \lim_{x\rightarrow1}v(x,t)  \label{compatibility}
\end{align}
for all $t\geq 0$.

We assume that $\Lambda$, $F$ and $g^u$ are uniformly Lipschitz-continuous with respect to their arguments. Let
\begin{align}
l_{\Lambda} & \operatorname*{=~ess\,sup~max}_{x\in[0,1],{ w=(u\,v)^T}\in\R^2\phantom{aa}} \{ \|\partial_u \Lambda(x, w) \|_{\infty},\|\partial_v \Lambda(x, w)\|_{\infty} \},  \label{bound Lambda_w}\\
l_{F} &= \operatorname*{ess\,sup}_{x\in[0,1], w\in\R^2}  \|\partial_w F(x, w) \|_{\infty},  \label{boundF_w}\\
l_{g^u} &= \operatorname*{ess\,sup}_{v\in\R,t\geq0}  |\partial_v g^u(v,t) |,  \label{bound g^u_v}
\end{align}
where $\partial$ denotes partial derivative, be the finite Lipschitz constants with respect to the state argument.
Consequently, $\Lambda$, $F$ and $g^u$ are bounded when evaluated along bounded trajectories. Moreover, we assume $\lambda^u$ and $\lambda^v$ are positive and $\Lambda^{-1}$ is bounded, globally on $[0,1]\times \R^2$, with
 \begin{align}
 l_{\Lambda^{-1}} &= \sup_{x\in[0,1],w\in\R^2}  \|\Lambda(x,w)^{-1} \|_{\infty}.  \label{bound Lambda-1}
 \end{align}
For stabilization, the coupling terms are assumed to satisfy
\begin{align}
F(x,0)&=0, & g^u(0,t)&=0  \label{F(0)=0}
\end{align}
for all $t\geq0$ and $x\in[0,1]$, i.e.,~the origin is assumed to be an equilibrium, although (\ref{F(0)=0}) need not be satisfied in the case that tracking is the objective (see also Section \ref{sec:tracking}).


\subsection{Characteristic lines}
\begin{figure}[htbp!]\centering
\includegraphics[width=.6\columnwidth]{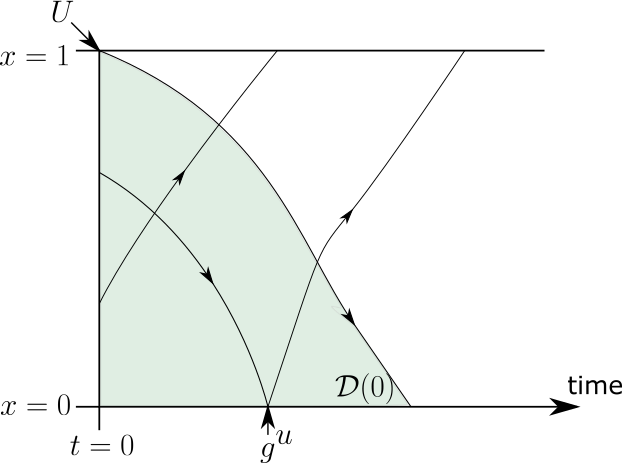}
\caption{Schematic of system (\ref{w_t})-(\ref{uvIC}), illustrating the characteristic lines of $u$ (positive ``upwards'' $x$-direction) and $v$ (negative ``downwards'' $x$-direction); the input $U$ entering at $x=1$; the boundary reflection term $g^u$ at $x=0$; and the determinate set $\D(0)$ (shaded in green). }
\label{fig:characteristic lines}
\end{figure}
The characteristic lines of system (\ref{w_t})-(\ref{uvIC}) are sketched in Figure \ref{fig:characteristic lines}.
 The characteristic lines passing through some point $(x,t)$ can be defined by 
\begin{align}
\begin{array}{rl}
\frac{d}{ds}\xi^u(x,t;s) &= \lambda^u\left(\xi^u(x,t;s),w(\xi^u(x,t;s),s)\right),  \\
 \xi^u(x,t;t)&=x,
\end{array}\label{xi^u}\\
\begin{array}{rl}
\frac{d}{ds}\xi^v(x,t;s) &= -\lambda^v\left(\xi^v(x,t;s),w(\xi^v(x,t;s),s)\right) , \\
 \xi^v(x,t;t)&=x,
\end{array}\label{xi^v}
\end{align}
where $\xi^u$ and $\xi^v$ are the locations that are reached at time $s$.
With $w\in\X^{c,c^{\prime}}_{[0,1]\times[0,T]}$ for some $c$, $c^{\prime}$ and $T\in\R\cup\{\infty\}$, the right-hand sides of (\ref{xi^u})-(\ref{xi^v}) are uniformly Lipschitz continuous in $\xi^u$ and $\xi^v$, respectively. Therefore, they have a unique solution that can be extended from $(x,t)$ until they reach the boundary of the domain $[0,1]\times[0,T]$.

For later use, also consider  the following alternative parameterization of the characteristic lines
\begin{align}
\tau^u(t;x) &= t+\int_0^x\frac{1}{\lambda^u(\xi,w(\xi,\tau^u(t;\xi))}d\xi \label{phi^u},\\
\tau^v(t;x) &= t+\int_x^1\frac{1}{\lambda^v(\xi,w(\xi,\tau^v(t;\xi))}d\xi, \label{phi^v} ,
\end{align}
where $\tau^u$ and $\tau^v$ are the times at which the characteristic lines starting at time $t$ and the spatial boundary $x=0$ and $x=1$, respectively, reach the location $x\in[0,1]$.

The dynamics of $u$ along $(\xi^u(x,t;s),s)$ and  $v$ along $(\xi^v(x,t;s),s)$, respectively, reduce to an ODE. For any $x,t\in[0,1]\times[0,\infty)$, using (\ref{xi^u})-(\ref{xi^v}) it follows that
\begin{align}
\frac{d}{ds}u(\xi^u(x,t;s),s) &= f^u(\xi^u(x,t;s),w(\xi^u(x,t;s),s)), \label{u_s}\\
\frac{d}{ds}v(\xi^v(x,t;s),s) &= f^v(\xi^v(x,t;s),w(\xi^v(x,t;s),s)).  \label{v_s}
\end{align}

As the second derivatives of the state are not well-defined for Lipschitz solutions, the dynamics of the time derivatives $u_t$ and $v_t$ need to be expressed as integral equations. 
Differentiating (\ref{w_t}) with respect to $t$, substituting the terms $u_x$ and $v_x$ (but not $u_{xt}$ and $v_{xt}$) by solving (\ref{w_t}) for $w_x$, and integrating along the characteristic lines, yields
\begin{align}
u_t(x,t) = & u_t^0(x,t) + \int_{t^u_0}^t \left[c_1\,(u_t)^2 + c_2\,u_t\, v_t \right. \nonumber \\
&\left. \quad + c_3\,u_t  + c_4\,v_t   \right](\xi^u(x,t;s),s)  \, ds\label{u_ts},\\
v_t(x,t) = & v_t^0(x,t) + \int_{t^v_0}^t \left[c_5\,u_t\, v_t  + c_6\,(v_t)^2 \right.    \nonumber \\
 &\left. \quad + c_7\,u_t  + c_8\,v_t   \right](\xi^v(x,t;s),s)  \, ds\label{v_ts},
\end{align}
where the functions $c_1$ through $c_8$, the initial/boundary data $u_t^0$ and $v_t^0$, and the times $t_0^u$ and $t_0^v$  are given in Appendix \ref{appendix A}. The shortened notation in the integrands means that all states are evaluated at $(\xi^u(x,t;s),s) $ and $(\xi^v(x,t;s),s) $, respectively, and $c_1$ through $c_8$ are evaluated at $\left(\xi^{u/v}(x,t;s),w(\xi^{u/v}(x,t;s),s) \right)$. Note that the right-hand sides of (\ref{u_ts})-(\ref{v_ts}) are quadratic in $w_t$. Therefore, the derivatives can blow up in finite time even if the state remains bounded.

\subsection{Well-posedness}
For existence of the solution in $\X^{c,c^{\prime}}_{[0,1]\times[0,T]}$ for any $T\in\R\cup\{\infty\}$, we need to ensure uniform boundedness of the Lipschitz constant $l(w)$ on  $[0,1]\times[0,T]$, which is implied whenever both partial derivatives $w_x$ and $w_t$ remain uniformly bounded.  For the spatial gradient $w_x$,
we have 
\begin{equation}
w_x(x,t) = \left(\Lambda(x,w(x,t))\right)^{-1}\left(w_t(x,t)-F(x,w(x,t)) \right) \label{w_x}.
\end{equation}
In view of the basic assumptions made in Section \ref{sec:assumptions}, (\ref{w_x}) implies that for bounded trajectories, the gradient $w_x$ is bounded if and only if $w_t$ is bounded. Thus, we can control the size of the gradient, $w_x$, via the size of state $w$, which is governed by (\ref{u_s})-(\ref{v_s}), and the  size of the time derivative $w_t$, which is governed by (\ref{u_ts})-(\ref{v_ts}).

In order to bound the growth of $w_t$, we will exploit the following  lemma  (Lemma 2.\Romannum{3} in \cite{cirina1970semiglobalsolutions}). It states, in a fashion similar to the classical Gronwall inequality,  that an integral inequality can be bounded by the solution of the corresponding equality.
\begin{lemma} \label{lemma bound quadratic}
Fix $T>0$ and let $\alpha$ be a real-valued function on $[0,T]$. Assume  there exists  $\gamma>0$ such that
\begin{align}
|\alpha(0)|&\leq e^{-\gamma T},\label{eq bound quadratic IC} \\
|\alpha(t)|&\leq |\alpha(0)| + \int_0^t \gamma\times\left(|\alpha(s)|^2+|\alpha(s)| \right)ds \label{eq bound quadratic}
\end{align}
for all $t\in[0,T]$. Then $\alpha(t)$ is bounded by 
\begin{equation}
|\alpha(t)|\leq \frac{|\alpha(0)|}{-|\alpha(0)|+(|\alpha(0)|+1)e^{-\gamma t}}\leq |\alpha(0)|e^{2\gamma T}.  \label{bound quadratic}
\end{equation}
\end{lemma}

The following Theorem is related to Theorems 3.\Romannum{2} and 3.\Romannum{4} in \cite{cirina1970semiglobalsolutions}, although the assumptions there are more restrictive, and also to the local result in \cite[Theorem 4.2]{li2001semiglobal}.
\begin{theorem}[Semi-global solution]\label{thm semiglobal solution}
Fix $T>0$ and let $\tilde{c}=\max\{\|w_0\|_{\infty},\,\|U\|_{\infty}\}$. { Assume\footnote{ The bound $\lambda^{\mathrm{max}}$ is only used to obtain a simple bound on the number of boundary reflections on an interval of arbitrary length $T$, when introducing the concept of semi-global solutions. It is not needed for the remainder of the paper.} further that there exists some finite  $\lambda^{\mathrm{max}}>0$} such that $\sup_{x\in[0,1], w\in\R^2}\|\Lambda(x, w)\|_{\infty}\leq \lambda^{\mathrm{max}}$. There exist  constants $\tilde{c}^{\prime}>0$ (see (\ref{ctildeprime estimate}) for  a conservative estimate ) and $c^{\prime}>0$, such that if $\|w_t(\cdot,0)\|_{\infty}\leq \tilde{c}^{\prime}$, $\|U_t\|_{\infty}\leq \tilde{c}^{\prime}$ and $\|\partial_t g^u\|_{\infty}\leq \tilde{c}^{\prime}$, then (\ref{w_t})-(\ref{uvIC}) has a unique solution up to time $T$ in $\X^{c,c^{\prime}}_{[0,1]\times[0,T]}$, where  $c =  \max\{1, l_{g^u} \} \exp(l_F T)\,\tilde{c}$. 
\end{theorem}
\begin{proof}
The proof follows the proofs of Theorem 3.\Romannum{2}  in \cite{cirina1970semiglobalsolutions} and Theorem 4.2 in \cite{li2001semiglobal}, which deal with classical $C^1$-solutions. However,  as in  \cite[Theorem 3.8]{bressan2000hyperbolic} one can show that if the initial and boundary values are Lipschitz continuous,  then the solution is also Lipschitz continuous. 

First, using Lipschitz-regularity of the initial condition and model data, local existence and uniqueness of the solution up to a potentially small time $T^{\prime}$ can be shown by following earlier results in \cite{douglis1952existence,li1985boundaryvalueproblems}. 

Exploiting that the right-hand side of (\ref{u_s})-(\ref{v_s}) is Lipschitz in the state, we have
\begin{align}
\left|\frac{d}{ds}u(\xi^u(x,t;s),s)\right| & \leq  l_F \|w(\cdot,s)\|_{\infty},   \label{u_s bound}\\
\left|\frac{d}{ds}v(\xi^v(x,t;s),s)\right| & \leq  l_F \|w(\cdot,s)\|_{\infty}.  \label{v_s bound}
\end{align}
Thus, the state in the interior of the domain  can grow at most exponentially with time. The boundary reflection term $g^u$ is also Lipschitz in the state and, because of the bound on $\|\Lambda(x, w)\|_{\infty}$, on each sub-interval of time of length $\frac{2}{\lambda^{\mathrm{max}}}$, the characteristic line on which  the supremum of the state over that interval is attained can have at most one boundary reflection, at which the state-norm is multiplied by at most the factor $\max\{1,l_{g^u}\}$.
 In summary,  if $w$ exists up to time $T$, then 
\begin{equation}
\|w(\cdot,t)\|_{\infty}\leq c =  \max\{1, l_{g^u} \}^{\left\lceil \frac{\lambda^{\mathrm{max}}\,T}{2} \right\rceil} e^{l_F T}\,\tilde{c}
\end{equation}
for all $t\leq T$. 

To bound the time-derivatives, with reference to (\ref{u_ts})-(\ref{v_ts}) let
\begin{align}
\gamma_1& = \sup_{x\in[0,1],\,\|z\|\leq c}\max\left\{|c_1(x,z)+c_2(x,z)|, |c_3(x,z)+\right. \nonumber\\
+c&_4(x,z)|,\left.|c_5(x,z)+c_6(x,z)|,|c_7(x,z)+c_8(x,z)| \right\}\label{gamma_1}  .
\end{align}
It follows that
\begin{align}
|u_t(x,t)| &\leq  |u_t^0(x,t)| + \int_{t^u_0}^t \gamma_1\times \left(\|w_t(\xi^u(x,t;s),s)\|_{\infty}^2 \right.    \nonumber \\
 &\left. + \|w_t(\xi^u(x,t;s),s)\|_{\infty} \right) ds,\\
|v_t(x,t)| & \leq |v_t^0(x,t)| + \int_{t^v_0}^t  \gamma_1\times \left(\|w_t(\xi^v(x,t;s),s)\|_{\infty}^2 \right.    \nonumber \\
 &\left. + \|w_t(\xi^v(x,t;s),s)\|_{\infty} \right) \, ds.
\end{align}
The initial/boundary terms $u_t^0(x,t)$ and $v_t^0(x,t)$ can be bounded via $\tilde{c}^{\prime}$ as follows. In case of $t^u_0(x,t)=0$ and $t^v_0(x,t)=0$, i.e., the characteristic line originates at $t=0$, then $|u_t^0(x,t)|\leq \tilde{c}^{\prime}$ and $|v_t^0(x,t)|\leq \tilde{c}^{\prime}$   because $\|w_t(\cdot,0)\|_{\infty}\leq \tilde{c}^{\prime}$ by  assumption. If $t^v_0(x,t)>0$, then $|v_t^0(x,t)| \leq \| U_t\|_{\infty}\leq \tilde{c}^{\prime}$  by assumption. If $t^u_0(x,t)>0$, the boundary condition at $x=0$ satisfies
\begin{equation}\begin{aligned}
|u_t^0&(x,t)|= |u_t(0,t^u_0(x,t))|  = |\partial_t g^u(v(0,t^u_0),t^u_0) \\
 +& \partial_v g^u(v(0,t^u_0),t^u_0)\,v_t(0,t^u_0)| \leq \| \partial_t g^u \|_{\infty} + l_{g^u} |v_t(0,t^u_0)|, 
\label{eq bound u_t(0)}
\end{aligned}\end{equation}
where $\partial_t g^u$ is also bounded by $ \tilde{c}^{\prime}$ by assumption. Note again that  for each subinterval of length $\frac{2}{\lambda^{\mathrm{max}}}$, the characteristic line on which  the supremum of $w_t$ over that subinterval is attained can have at most one boundary reflection. Therefore, it can happen at most $\left\lceil \frac{\lambda^{\mathrm{max}}\,T}{2} \right\rceil$ times that $\|w_t(\cdot,t)\|_{\infty}$ is multiplied by the factor $\max\{1,l_{g^u}\}$ and added to  $\|\partial_t g^u\|_{\infty}$. Setting $ \tilde{c}^{\prime}$ to the conservative value
\begin{equation}
 \tilde{c}^{\prime} = \frac{1}{\sum_{i=0}^{\left\lceil \frac{\lambda^{\mathrm{max}}\,T}{2} \right\rceil} \max\{1,l_{g^u}\}^i}\,e^{-2\gamma_1T}, \label{ctildeprime estimate}
\end{equation}
 we can apply Lemma \ref{lemma bound quadratic} for $\alpha(t)=\|w_t(\cdot,t)\|_{\infty}$ on each subinterval in between the times when $\|w_t(\cdot,t)\|_{\infty}$ can be amplified at the boundary $x=0$ (the 2 in the exponent of (\ref{ctildeprime estimate}), as opposed to (\ref{eq bound quadratic IC}), is needed to apply Lemma \ref{lemma bound quadratic} successively), to obtain 
\begin{equation}\begin{aligned}
\|w_t(\cdot,t)\|_{\infty} \leq & \max\{w_t(\cdot,t)_{\infty},\|U_t\|_{\infty},\|\partial_t g^u\|_{\infty}\}\\
& \times \sum_{i=0}^{\left\lceil \frac{\lambda^{\mathrm{max}}\,T}{2} \right\rceil} \max\{1,l_{g^u}\}^i\, e^{2\gamma_1T}
\end{aligned}\end{equation} 
for all $t\leq T$. Importantly, boundedness of  $\|w_t(\cdot,t)\|_{\infty}$ does not depend on $T^{\prime}$ but is valid uniformly  up to time $T$. The spatial gradient $w_x$, and thus the Lipschitz constant $l(w)$, can now be bounded uniformly up to time $T$  via (\ref{w_x}) and the previously established bounds on $w$ and $w_t$. As such, the local existence and uniqueness result can be applied iteratively on the intervals $[0,T^{\prime}]$, $[T^{\prime},2T^{\prime}]$, $\ldots$, up to time $T$.
\end{proof}
\begin{figure}[htbp!]\centering
\includegraphics[width=.7\columnwidth]{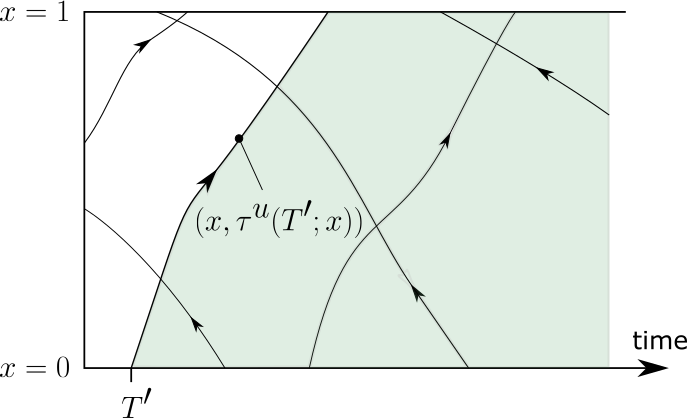}
\caption{Schematic of the characteristic lines of transformed system (\ref{w_x}), and the determinate set (\ref{det set 2})  for some $T^{\prime}\geq 0$ (shaded green).}
\label{fig:determinate set 2}
\end{figure}
Using the concept of semi-global solutions, we can obtain the following sufficient condition for global  (in time) existence of a solution to (\ref{w_t})-(\ref{uvIC}). Specifically, we can reverse the roles of $x$ and $t$ and look at (\ref{w_x}) as an initial-boundary value problem in the positive $x$-direction, with initial condition at $x=0$, a boundary condition for $u$ at $t=0$ and an open boundary for $v$ at $t=\infty$. See  Figure \ref{fig:determinate set 2}. 
For each  $(x,t)\in[0,1]\times[0,\infty)$, the dynamics along the characteristic lines passing through $(x,t)$ in the positive $x$-direction satisfy
\begin{align}
\frac{d}{d\xi}u(\xi,\tau^{u,x}(x,t;\xi)) & = \phantom{-}\frac{f^u\left(x,w(x,t) \right) }{\lambda^u(x,w(x,t))}, \label{u_x 1}\\
\frac{d}{d\xi}v(\xi,\tau^{v,x}(x,t;\xi)) &=  -\frac{f^v\left(x,w(x,t) \right) }{\lambda^v(x,w(x,t))},  \label{v_x 1}
\end{align}
and the time derivatives satisfy the  integral equations
\begin{align}
u_t(x,t) = & u_t^{0,x}(x,t) + \int_{x^u_0(x,t)}^x  \frac{1}{\lambda^u} \left[{c}_1(u_t)^2 + {c}_2 u_t v_t \right. \nonumber \\
& \qquad \left. +{c}_3 u_t  +{c}_4 v_t   \right](\xi,{\tau}^{u,x}(x,t;\xi))  \, d\xi, \label{u_t integral2}\\
v_t(x,t) = & v_t^{0,x}(x,t) - \int_{0}^x  \frac{1}{\lambda^v} \left[{c}_5 u_t v_t + {c}_6 (v_t)^2 \right. \nonumber \\
& \qquad \left. + {c}_7 u_t  + {c}_8 v_t   \right](\xi,{\tau}^{v,x}(x,t;\xi))  \, d\xi, \label{v_t integral2}
\end{align}
where the characteristic lines are parameterized as
\begin{align}
\tau^{u,x}(x,t;\xi) &= t-\int_{\xi}^x\frac{1}{{\lambda}^u(y,w(y,\tau^{u,x}(x,t;y)))}dy, \label{phi^u start at 0} \\
\tau^{v,x}(x,t;\xi) &= t-\int_{\xi}^x\frac{1}{{\lambda}^v(y,w(y,\tau^{v,x}(x,t;y)))}dy, \label{phi^v start at 0} 
\end{align}
$(\xi^{u,x}_0(x,t),t^{u,x}_0(x,t))$ is defined as the intersection of $(\xi,\tau^{u,x}(x,t;\xi))$ with either $[0,1]\times 0$ or $0\times [0,\infty)$, and the initial values are 
\begin{align}
u_t^{0,x}(x,t) &=  \begin{cases} u_t(\xi^{u,x}_0(x,t),0) & \text{ if } t^{u,x}_0(x,t)=0  \\ 
											u_t(0,t^{u,x}_0(x,t)) & \text{ if } \xi^{u,x}_0(x,t) = 0  \end{cases},  \\
v_t^{0,x}(x,t) &= v_t(0,\tau^{v,x}(x,t;0)).
\end{align}
\begin{theorem}[Sufficient condition for global solution] \label{thm: global existence condition}
Let $\tilde{c}$ be an upper bound on both $\|w_0\|_{\infty}$ and $\|v(0,\cdot)\|_{\infty}$. There exist constants $\tilde{c}^{\prime}>0$ (see (\ref{ctildeprime estimate 2}) for a conservative estimate) and $c^{\prime}>0$, such that if  $\|u_t(\cdot,0)\|_{\infty}\leq \tilde{c}^{\prime}$, $\|\partial_t g^u\|_{\infty}\leq \tilde{c}^{\prime}$ and    $\|v_t(0,\cdot)\|_{\infty}\leq \tilde{c}^{\prime}$, then (\ref{w_t})-(\ref{uvIC}) has a unique global solution in $\X^{\kappa_1\tilde{c},c^{\prime}}_{[0,1]\times [0,\infty)}$ where
\begin{equation}
\kappa_1 = \max\{1, l_{g^u} \} \exp(l_F l_{\Lambda^{-1}}). \label{kappa_1}
\end{equation}
\end{theorem}
\begin{proof}
Because of (\ref{uBC}) and the assumptions on $g^u$,  $\|w(0,t)\|_{\infty}\leq \max\{1,l_{g^u}\}|v(0,t)|$. Equations (\ref{u_x 1})-(\ref{v_x 1}) imply
 \begin{align}
\left|\frac{d}{d\xi}u(\xi,\tau^{u,x}(x,t;\xi))\right| &\leq l_{\Lambda^{-1}}l_F \|w(x,\cdot)\|_{\infty}, \label{u_x bound}\\
\left|\frac{d}{d\xi}v(\xi,\tau^{v,x}(x,t;\xi)) \right| &\leq l_{\Lambda^{-1}}l_F \|w(x,\cdot)\|_{\infty},  \label{v_x bound}
\end{align}
which, combined with the bounds $\|u(\cdot,0)\|_{\infty}\leq \tilde{c}$ and $\|w(0,t)\|_{\infty}\leq \max\{1,l_{g^u}\}\tilde{c}$,  gives the bound $\kappa_1\tilde{c}$ on the state norm.

To bound the time-derivatives, let 
\begin{align}
\gamma_2& = \sup_{ x\in[0,1],\, \|z\|\leq \kappa_1\tilde{c}} \max \left\{ \left|\frac{c_1+c_2}{\lambda^u}\right|,\left|\frac{c_3+c_4}{\lambda^u}\right|, \nonumber \right. \\
& \hspace{2cm} \left.\left|\frac{c_5+c_6}{\lambda^v}\right|,\left|\frac{c_7+c_8}{\lambda^v}\right| \right\}, \label{gamma_2} \\
\intertext{where the arguments $(x,z)$ of $c_{\cdot}$ and $\lambda^{u/v}$ are omitted, and}
 \tilde{c}^{\prime}& = \frac{1}{2\,\max\{1,l_{g^u}\}}\,e^{-\gamma_2}. \label{ctildeprime estimate 2}
\end{align}
Differentiating (\ref{uBC}) w.r.t. $t$ yields $\|u_t(0,\cdot)\|_{\infty}\leq e^{-\gamma_2}$. Thus, $\|w_t(0,\cdot)\|_{\infty}\leq  e^{-\gamma_2}$,  $\|u_t(\cdot,0)\|_{\infty}\leq  e^{-\gamma_2}$ by assumption and, as in the proof of Theorem \ref{thm semiglobal solution}, we can apply Lemma \ref{lemma bound quadratic} to (\ref{u_t integral2})-(\ref{v_t integral2}) to exclude blow-up of the time-derivatives, with the boundary $x=1$ playing the role of the  terminal time $T$.
\end{proof}

\subsection{Determinate sets}
The control input $U(t)$ entering at $x=1$  propagates through the spatial domain $[0,1]$ with finite speed $\lambda^v$ along the characteristic line $(\xi^v(1,t;s),s)$. See  Figure \ref{fig:characteristic lines}. Therefore, the state in the interior of the domain is affected after a delay.  As such, the state in the interior of the domain is independent of the control input for some time, and can be predicted based on the state at time $t$ alone. A time-space domain on which the solution is independent of the control input has been called a determinate set \cite{li2000semi}, or also a domain of determinacy \cite{bressan2000hyperbolic}.

For $t\geq 0$, the maximal determinate set is given by
\begin{equation}
\D(t) = \left\{(x,s):\,x\in[0,1],\,s\in[t,\tau^v(t;x)]\right\}.  \label{eq: D}
\end{equation}
For $t=0$, the following theorem formalizes aspects of  the above remarks, i.e., that the solution on the maximum determinate set $\D(0)$, as well as the domain $\D(0)$ itself, is uniquely defined by the initial condition $w_0$ and independent of the control input $U(t)$ for $t> 0$ (the input $U(0)$ is uniquely related to $w_0$ by the compatibility condition (\ref{compatibility})). Compare also with  \cite[Lemma 3.1]{li2000semi} and  \cite[Theorem 2]{strecker2017output}.
\begin{theorem}\label{thm:determinate set 1}
With reference to (\ref{gamma_1}) and (\ref{kappa_1}), let $c= \kappa_1\,\|w_0\|_{\infty}$ and $\tilde{c}^{\prime}=\frac{1}{2\max\{1,l_{g^u}\}}e^{-2\gamma_1l_{\Lambda^{-1}}}$.  If $\|w_t(\cdot,0)\|_{\infty}\leq \tilde{c}^{\prime}$ and $\|\partial_t g^u\|_{\infty}\leq \tilde{c}^{\prime}$, then the Cauchy problem consisting of (\ref{w_t}), (\ref{uBC}) and (\ref{uvIC}) has a unique solution on the maximal determinate set $\mathcal{D}(0)$ and this solution lies in $\X^{c,c^{\prime}}_{\D(0)}$ for some $c^{\prime}>0$.  Moreover, $\tau^v(0;\cdot)$, $w(\cdot,\tau^v(0;\cdot))$ and $u_t(\cdot,\tau^v(0;\cdot))$, but not $v_t(\cdot,\tau^v(0;\cdot))$, are independent of the control input.  Moreover,
\begin{equation}\begin{aligned}
\operatorname*{ess\,sup}_{(x,t)\in\D(0)}\|w_t(x,t)\|_{\infty} \leq& (\max\{1,l_{g^u}\}\|w_t(\cdot,0)\|_{\infty} \\
&+ \|\partial_t g^u\|_{\infty} )\times e^{2\gamma_1l_{\Lambda^{-1}}}.
\end{aligned} \label{bound w_t on D(0)}\end{equation}
 \end{theorem} 
\begin{proof}
Existence of the solution and the bound on the time derivative follows as in Theorem \ref{thm semiglobal solution},  with the terminal time bounded as $T\leq l_{\Lambda^{-1}}$. Independence of the solution from $U$ follows by the fact that for all points $(x,t)\in\D(0) \setminus\{(\cdot,\tau^v(0;\cdot))\}$, the characteristic lines passing through $(x,t)$ originate from the initial condition at $t=0$ or the boundary condition for $u$ at $x=0$, but not from the boundary at $x=1$ where the input enters.   The state $w(\cdot,\tau^v(0;\cdot))$ is also uniquely defined by $w_0$ due to Lipschitz-continuity. The derivative $u_t(\cdot,\tau^v(0;\cdot))$ is also independent of the choice of $U$ because for points on this line, the entire integration path in (\ref{u_ts}), except for one point, lies in $\D(0)\setminus\{(\cdot,\tau^v(0;\cdot))\}$ (see also Figure \ref{fig:characteristic lines}, and Appendix A.3 in \cite{strecker2017output} for a similar situation). However,  $v_t(\cdot,\tau^v(0;\cdot))$ is determined by $U_t(0)$ because (\ref{v_ts}) is integrated along the line $(\cdot,\tau^v(0;\cdot))$.  
\end{proof}

Using the concept of determinate sets, we can  formulate the following condition for convergence to the origin.
\begin{theorem}\label{thm:determinate set 2}
Assume $v(0,\cdot)\in\X^{\tilde{c},\tilde{c}^{\prime}}_{[0,\infty)}$ and  $w_0\in\X^{\tilde{c},\tilde{c}^{\prime}}_{[0,1]}$  with $\tilde{c}$ and $\tilde{c}^{\prime}$ such that, as in Theorem \ref{thm: global existence condition}, the solution exists on $[0,1]\times[0,\infty)$.   Then,  for all $T^{\prime}\geq0$, 
\begin{equation}
\sup_{x\in[0,1],\,t\geq \tau^u(T^{\prime};x)} \|w(x,t)\|_{\infty} \leq \gamma \times \sup_{t\geq T^{\prime}}  |v(0,t)|,
\end{equation}
where  $\gamma=\max\{1, l_{g^u} \} \exp(l_F l_{\Lambda^{-1}})$. 
In particular, if $v(0,t)=0$ for all $t\geq T^{\prime}$ then $w(x,t)\equiv 0$ for all $x\in[0,1]$, $t\geq \tau^u(T^{\prime};x)$.
\end{theorem}
\begin{proof}
 The domain (see Figure \ref{fig:determinate set 2})
\begin{equation}
 \{(x,t):\,x\in[0,1],\,t\geq \tau^u(T^{\prime};x)\}  \label{det set 2}
\end{equation} 
  is the maximal determinate set of the Cauchy problem (solved in the positive $x$-direction) consisting of (\ref{w_x}) and initial data $w(0,t)=\left(v(0,t),g^u(v(0,t)) \right)$ restricted to $t\geq T^{\prime}$. With this, the derivation of the bound follows as in the proof of Theorem \ref{thm: global existence condition}. 
 %
\end{proof}

\section{Exact state feedback control} \label{sec: exact state feedback}
\subsection{Outline} \label{sec: control outline}
\begin{figure}[htbp!]\centering
\includegraphics[width=.9\columnwidth]{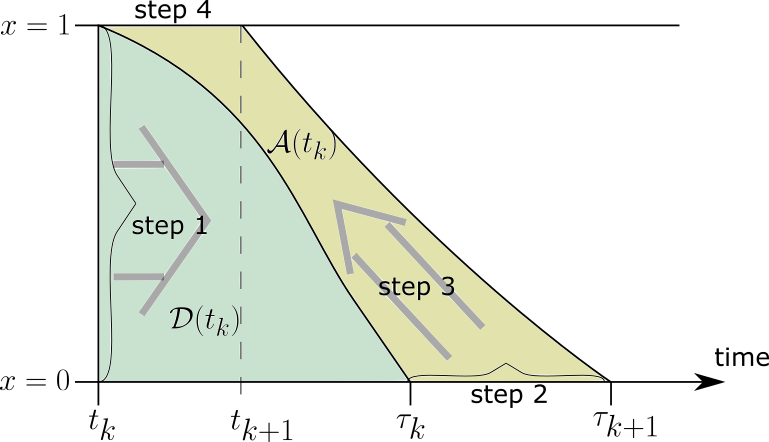}\\
\caption{Outline of the state-feedback control method at time $t=t_k$, $k\in\N$: 1) predict state on determinate set $\D(t_k)$ (shaded in green); 2)  design target boundary values for $v(0,t)$, $t\geq\tau_k$; 3) solve target dynamics   \emph{backwards} along the characteristic lines over $\mathcal{A}(t_k)$ (yellow); 4) set input to boundary value of target dynamics.  }
\label{fig:outline control}
\end{figure}
We present a sampled-time control law with sampling period $\theta>0$. At each sampling instance $t_k=k\,\theta$, $k\in\N$, the inputs are computed for the time interval $[t_k,t_{k+1}]$. The control law is based on  the following observations: 
\begin{enumerate}
\item  As stated in Theorem \ref{thm:determinate set 1}, for any $t\geq 0$ the state on the determinate set $\D(t)$  is uniquely determined by the state at that time, $w(\cdot,t)$, and cannot be affected by control input $U(s)$ for $s\geq t$. 
\item By virtue of Theorems \ref{thm: global existence condition} and \ref{thm:determinate set 2}, one can ensure existence of a global solution and convergence to the origin, respectively, by ``slowly'' steering $v(0,\cdot)$, i.e. the boundary value at the uncontrolled boundary, to zero.
\end{enumerate} 
Because of 1), rather than trying to control the state at the current time, we design the inputs to control the future state on the characteristic lines along which the effect of the control inputs propagate through the spatial domain. Exploiting the  observations in 2), we first construct desirable boundary values for $v(0,\cdot)$, which also need to be compatible with the predictions over the determinate set $\D(t)$. Then, the actual inputs $U$ that achieve these desired boundary traces are computed by solving the system dynamics in the positive $x$-direction, i.e.,~\emph{backwards} relative to the propagation direction of the input, with the designed target boundary values for $v(0,\cdot)$ entering via the initial condition at $x=0$.  See  Figure \ref{fig:outline control}.

For sampling instant $t_k$, $k\in\N$,  define the time at which the control input $U(t_k)$ reaches the uncontrolled boundary at $x=0$ by
\begin{equation}
\tau_k = \tau^v(t_k;0).  \label{tau_k}
\end{equation}

\subsection{State prediction}
Let $\tilde{w}^{k}$ denote the state predictions at the $k$-th sampling instant, $k\in\N$.
Assuming no model uncertainty and exact state measurements and actuation, $\tilde{w}^{k}(\cdot,\cdot)$ satisfies
\begin{align}
\tilde{w}_t^{k}(x,t) &= \Lambda(x,\tilde{w}^{k}(x,t))\,\tilde{w}_x^{k}(x,t)  + F(x,\tilde{w}^{k}(x,t)), \label{wtilde_t}  \\
\tilde{u}^{k}(0,t) &= g^u(\tilde{v}^{k}(0,t),t),  \label{utildeBC}\\
\tilde{v}^{k}(1,t)  & = U(t),\label{vtildeBC} \\
\tilde{w}^{k}(x,t_k) &= W_k(x), \label{wtildeIC} 
\end{align}
for $t\geq t_k$, $x\in[0,1]$, with the state measurement at time $t_k$ 
\begin{equation}
W_k = w(\cdot,t_k).   \label{measurement exact state feedback}
\end{equation}
Without model and measurement errors, the predicted trajectory $\tilde{w}$ is equal to the actual trajectory $w$. Therefore, $\tau^v(t_k,\cdot)$ and $\D(t_k)$ are also the characteristic lines and determinate set of $\tilde{w}$, respectively. Consequently, $ \tilde{w}(t_k;\cdot,\cdot)$ can be predicted over the whole maximum determinate set $\D(t_k)$ based on the state $w(\cdot,t_k)$ alone, i.e.~independently of the control input $U$, simply by solving (\ref{wtilde_t})-(\ref{wtildeIC}) over the domain $\D(t_k)$. 
 Out of the prediction over all of $\D(t_k)$, the following  predictions are needed for implementation of the control law:
\begin{equation}
w(\cdot,t_k) \mapsto \left(\begin{matrix} \tilde{w}^k(\cdot,\tau^v(t_k;\cdot)) \\ \tilde{u}_t^k(\cdot,\tau^v(t_k;\cdot)) \\ \tau^v(t_k;\cdot)  \end{matrix} \right).  \label{prediction k}
\end{equation}

\subsection{Target boundary values at $x=0$}

As indicated under point 2) in Section \ref{sec: control outline}, it would be straightforward to control  the system via the boundary value $v(0,\cdot)$. Following  ideas from \cite{strecker2017output}, we introduce the virtual control input $U^*$, which is the desired value for $v(0,\cdot)$. 
By Theorem \ref{thm:determinate set 1}, $v(0,t)$ for $t<\tau_k$ is not affected by $U(t)$ for $t\geq t_k$. Instead, we can design the inputs $U(t)$ over the interval $t\in[t_k,t_{k+1}]$ to control the boundary values $v(0,t)$ over the interval $t\in[\tau_k,\tau_{k+1}]$. The following is required  of $U^*$: 
\begin{enumerate}
\item  it is a Lipschitz-continuous continuation of the prediction $\tilde{v}(t_k;0,t)$, $t\leq \tau_k$, to ensure $\X$-compatibility;
\item  it has small time-derivative (a.e.) to ensure existence of solution as in Theorem \ref{thm: global existence condition};
\item it converges to zero in finite or even minimum time; 
\item  the absolute value should not grow beyond some bound depending on the initial condition,  so that the solution remains bounded during transients. 
\end{enumerate} 
 One design that satisfies all four conditions and uses only data  that is available or predictable at time $t_k$ is
 \begin{equation}
U^{*,k}(t)= \begin{cases} v_0^k\cdot \left(1-\frac{\delta\cdot(t-\tau_k)}{|v_0^k|} \right) & t\in\left[\tau_k,\tau_k+\frac{|v_0^k|}{\delta}\right]\\
 0 &  t> \tau_k+\frac{|v_0^k|}{\delta}
 \end{cases},\label{U^*}
 \end{equation}
 where $v_0^k=\tilde{v}(t_k;0,\tau_k)$ and $\delta$ is the desired  bound on $|v_t(0,\cdot)|$.

\subsection{Construction of control inputs}

{
Define the future target state  on the characteristic line along which the control input propagates, and its derivative as
\begin{align}
\tilde{w}^{*,k}(x,t)&= \left(\tilde{u}^{*,k}(x,t)~\tilde{v}^{*,k}(x,t)\right) = \tilde{w}^k(x,\tau^v(t;x)),\\
\tilde{w}^{*,k}_{\partial t}(x,t)&= \left(\tilde{u}^{*,k}_{\partial t}(x,t)~\tilde{v}^{*,k}_{\partial t}(x,t)\right) = \tilde{w}_t^k(x,\tau^v(t;x)).
\end{align}
 The subscript $\partial t$ notation introduced here reflects the relationship to the time-derivative, while emphasizing that, in general,  $\partial_t \tilde{w}^{ *,k} \neq \tilde{w}^{ *,k}_{\partial t}$  (see also (\ref{target 1})).

The control inputs which ensure   $v(0,t)=U^{*,k}(t)$ for all $t\in[\tau_k,\tau_{k+1}]$ can be constructed by solving the following system over the domain $[0,1]\times[t_k,t_{k+1}]$:
\begin{align}
\partial_t \tilde{u}^{*,k}(x,t) =& \partial_t \tau^v(t;x)\times \tilde{u}^{*,k}_{\partial t}(x,t), \label{target 1}\\
{ \partial_x \tilde{v}^{*,k}(x,t) =}&{ -\frac{f^v(x,\tilde{w}^{*,k}(x,t))}{\lambda^v(x,\tilde{w}^{*,k}(x,t))} ,}\label{target 2}\\
\partial_t \tilde{u}^{*,k}_{\partial t}(x,t) =& -\mu(x,\tilde{w}^{*,k}(x,t),t) \times\partial_x\tilde{u}^{*,k}_{\partial t}(x,t) \nonumber \\
&  + \nu(x,\tilde{w}^{*,k}(x,t),t) \times \left[c_1 \, (\tilde{u}^{*,k}_{\partial t})^2 \right. \nonumber \\
&\left. ~~~~+ c_2\,\tilde{u}^{*,k}_{\partial t}\,\tilde{v}^{*,k}_{\partial t}+c_3 \,\tilde{u}^{*,k}_{\partial t}+c_4\,\tilde{v}^{*,k}_{\partial t}\right],\label{target 3}\\
\partial_x \tilde{v}^{*,k}_{\partial t}(x,t)= & -\frac{1}{\lambda^v(x,\tilde{w}^{*,k}(x,t))} \times \left[c_5\, \tilde{u}^{*,k}_{\partial t}\,\tilde{v}^{*,k}_{\partial t}  \right.\nonumber \\
& ~~~~ \left.+ c_6\,(\tilde{v}^{*,k}_{\partial t})^2 + c_7\,\tilde{u}^{*,k}_{\partial t} + c_8\,\tilde{v}^{*,k}_{\partial t} \right] \label{target 4},\\
\partial_t \tau^v(t;x) =& 1 - \int_x^1\frac{\partial_u \lambda^v\times \tilde{u}^{*,k}_{\partial t} + \partial_v \lambda^v\times \tilde{v}^{*,k}_{\partial t}}{\left(\lambda^v(\xi,\tilde{w}^{*,k}(\xi,t))\right)^2} d\xi,
\end{align}
with $c_1$ through $c_8$ as in Appendix \ref{appendix A}, and 
\begin{align}
\nu(x,z,t) =& \frac{\partial_t \tau^v(t;x) \times \lambda^v(x,z)}{\lambda^u(x,z)+\lambda^v(x,z)}, \label{nu} \\
\mu(x,z,t) =& \nu(x,z,t)\times \lambda^u(x,z), \label{mu} 
\end{align}
and boundary and initial conditions 
\begin{align}
{ \tilde{v}^{*,k}(0,t) =}& { U^{*,k}(\tau^v(t;0)), }  \label{target vBC} \\ 
\tilde{v}^{*,k}_{\partial t}(0,t) =& \partial_t U^{*,k}(\tau^v(t;0)),   \label{target v_t BC}\\
\tilde{u}^{*,k}_{\partial t}(0,t) =&  \partial_v g^u(U^{*,k}(\tau^v(t;0)),\tau^v(t;0)) \times \tilde{v}^{*,k}_{\partial t}(0,t) \nonumber \\
& + \partial_t g^u(U^{*,k}(\tau^v(t;0)),\tau^v(t;0)),  \label{target uBC}\\
\tilde{u}^{*,k}(\cdot,t_k) =& \tilde{u}^k(\cdot,\tau^v(t_k;\cdot)),\\
\tilde{u}_{\partial t}^{*,k}(\cdot,t_k) =& \tilde{u}_t^k(\cdot,\tau^v(t_k;\cdot)). \label{target last}
\end{align}
Equations (\ref{target 1})-(\ref{mu}) are a copy of the dynamics of $w$ along the characteristic line $(\cdot,\tau^v(t;\cdot))$.
Their derivation uses (\ref{u_ts})-(\ref{v_ts}) and the approach from the proof  of Theorem 2 in \cite{strecker2017output}. Equations (\ref{target 3})-(\ref{target 4}) are here written as PDEs, but should strictly speaking be interpreted as integral equations like (\ref{u_ts}) and (\ref{v_t integral2}).    Note that { (\ref{target 2}) and (\ref{target 4})  are ODEs}  in the $x$-direction with no time-dynamics. { Also note that in  (\ref{target v_t BC}),  $\partial_t U^{*,k}(\tau^v(t;0))$ is the partial derivative of $U^{*,k}$ with respect to time evaluated at time $\tau^v(t;0)$, not the total derivative of $U^{*,k}(\tau^v(t;0))$ with respect to $t$}.

Analogously to Theorem \ref{thm:determinate set 1}, for every $k\in\N$, the boundary value problem (\ref{target 1})-(\ref{target last}) has a unique solution  on the determinate set $[0,1]\times[t_k,t_{k+1}]$ which, in terms of the original coordinates, corresponds to the domain
\begin{equation}
 \mathcal{A}(t_k)=\left\{(x,t):\,x\in[0,1],\,t\in[\tau^v(t_{k};x),\tau^v(t_{k+1};x)]\right\}.  \label{eq: A}
\end{equation}
That is, the solution satisfying $v(0,t)=U^{*,k}(t)$, $t\in[\tau_k,\tau_{k+1}]$,  is equal to $\tilde{w}^{*,k}$ on $\mathcal{A}(t_k)$. Therefore, the closed-loop solution of (\ref{wtilde_t})-(\ref{wtildeIC}) satisfies 
\begin{equation}
v(0,t) = \tilde{v}^{*,k}(0,t) =  U^{*,k}(t)   \label{wtilde=wtildetarget}
\end{equation}
 for all $t\in[t_k,\tau_{k+1}]$ if and only if 
\begin{equation}
v(1,t) = U(t)=\tilde{v}^{*,k}(1,t) \label{U(t) exact}
\end{equation}
for all $t\in[t_k,t_{k+1}]$.
That is, by  solving (\ref{target 4}) in the positive $x$-direction, i.e., backwards relative to  the propagation direction of the input, we can compensate (potentially destabilizing) source terms to obtain the control inputs as in (\ref{U(t) exact}) that achieve the target boundary values (\ref{wtilde=wtildetarget}).
}

\subsection{Control law and finite time convergence}
Following the preparations above, the control inputs over each sampling interval are computed by the following algorithm (see also Figure \ref{fig:outline control}).

\begin{algorithm}[H]
\begin{algorithmic}[1]
\REQUIRE sampling instant $t_k$, state $w(\cdot,t_k)$,  $\delta$  \\
\ENSURE control input $U(t)$ for $t\in[t_k,t_{k+1}]$\\ ~

\STATE predict  states in (\ref{prediction k})  by solving (\ref{wtilde_t})-(\ref{wtildeIC})  over determinate set $\D(t_k)$
\STATE set $U^{*,k}$ as in (\ref{U^*}).
\STATE solve  target dynamics (\ref{target 1})-(\ref{target last}) over $[0,1]\times[t_k,t_{k+1}]$
\STATE set $U(t)$, $t\in[t_k,t_{k+1}]$, as per (\ref{U(t) exact})
\end{algorithmic}
\caption{Control algorithm }
\label{algorithm_statefeedback}
\end{algorithm}

Assuming exact model knowledge, Algorithm \ref{algorithm_statefeedback} steers the system to the origin in finite time, as stated in the following theorem.

\begin{theorem} \label{thm state feedback exact}
There exist constants $\tilde{c}^{\prime}>0$, ${c}^{\prime}>0$ and $\delta^{\mathrm{max}}>0$, which depend on $\|w_0\|_{\infty}$, such that if $\|w_t(\cdot,0)\|_{\infty}\leq  \tilde{c}^{\prime}$, $\|\partial_t g^u\|_{\infty}\leq  \tilde{c}^{\prime}$ and $\delta\leq \delta^{\mathrm{max}}$, then (\ref{w_t})-(\ref{uvIC}), in closed loop with $U$  constructed according to Algorithm \ref{algorithm_statefeedback} has a unique global solution in $\X^{c,c^{\prime}}_{[0,1]\times[0,\infty)}$, with  $c=\kappa_1^2\|w_0\|_{\infty}$ and $\kappa_1$ as in (\ref{kappa_1}), that satisfies 
\begin{align}
w(\cdot,t)&=0 &&\text{for all} & t&\geq \tau^u\left(\tau_0+\frac{|v(0,\tau_0)|}{\delta};1\right). \label{eq exact state feedback thm}
\end{align} 
\end{theorem}
\begin{proof} 
By construction, 
\begin{equation}
v(0,t) = U^{*,k}(t)  \label{v(0)=U*k}
\end{equation}
 for  $k=\left\lfloor \frac{t}{\theta} \right\rfloor$  and all $t\geq \tau_0$. Substituting  $t=t_{k+1}$ in (\ref{U^*}) for $k\in\N$, it is straightforward to see that 
\begin{equation}
U^{*, k+1}(s)=U^{*, k}(s)
\end{equation} 
for all  $s\geq \tau_{k+1}$. Therefore, 
\begin{equation}
v(0,t) = U^{*,0}(t)  \label{v(0)=U*0}
\end{equation}
for all $t\geq \tau_0$. 
Assuming existence of the solution, convergence to the origin follows from $U^{*, 0}(t)=0$ for all $t\geq \tau_0+\frac{|v(0,\tau_0)|}{\delta}$ and Theorem \ref{thm:determinate set 2}.

We next show that the solution exists in $\X^{c,c^{\prime}}_{[0,1]\times[0,\infty)}$. By Theorem \ref{thm:determinate set 1}, $\|w(x,t)\|\leq \kappa_1\|w_0\|_{\infty}$ for all $(x,t)\in\D(0)$. Due to the design of $U^*$, this implies $\|v(0,\cdot)\|_{\infty}\leq \kappa_1\|w_0\|_{\infty}$. Using Theorem  \ref{thm: global existence condition}, we obtain the a-priori bound $\|w\|_{\infty}\leq \kappa_1^2\|w_0\|_{\infty}$. Equation (\ref{ctildeprime estimate 2}) in Theorem  \ref{thm: global existence condition} also provides the condition 
\begin{equation}
\|v(0,\cdot)\|_{\infty}\leq \tilde{c}^{\prime}_2= \frac{1}{2\max\{1,l_{g^u}\}}e^{-\gamma_2},
\end{equation}
where $\gamma_2$ must be computed according to (\ref{gamma_2}) with the supremum taken over  $\|z\| \leq \kappa_1^2\|w_0\|_{\infty}$. For $t\geq \tau_0$, $|v_t(0,t)|\leq \tilde{c}^{\prime}_2$ is ensured simply by setting $\delta^{\mathrm{max}}=\tilde{c}^{\prime}_2$. Using Equation (\ref{bound w_t on D(0)}), 
\begin{equation}
\tilde{c}^{\prime} = \frac{1}{4\max\{1,l_{g^u}\}^2}e^{-2\gamma_1\gamma_2}, \label{ctilde 3}
\end{equation}
where $\gamma_1$ must be computed according to (\ref{gamma_1}) with the supremum taken over $\|z\|\leq \kappa_1\|w_0\|_{\infty}$, is sufficient to ensure $|v_t(0,t)|\leq \tilde{c}^{\prime}_2$ for all $t<\tau_0$. Note that (\ref{ctilde 3}) is stronger than the bounds on $\|u_t(\cdot,0)\|_{\infty}$ and $\|\partial_t g^u\|_{\infty}$  needed to apply Theorem \ref{thm: global existence condition} in the step above.
\end{proof}

The sampling period $\theta$ does not affect performance in Theorem \ref{thm state feedback exact} because zero model uncertainty is assumed. However, it is shown in Section \ref{sec:robustness} that, at least in certain cases,   the choice of both  $\theta$ and $\delta$ can affect robustness margins.


\begin{remark}[Continuous-time control law]\label{remark continuous feedback}
While the sampled-time control law in Algorithm \ref{algorithm_statefeedback} is likely preferable in practical applications, it is possible to write an equivalent control law where the control law is evaluated continuously in time. Such a continuous-time control law could be implemented as follows:
\begin{algorithm}[H]
\begin{algorithmic}[1]
\REQUIRE time $t$, state $w(\cdot,t)$, $\delta$\\
\ENSURE derivative of control input $\partial_t U(t)$\\ ~

\STATE predict $\tilde{w}(\cdot,\tau^v(t;\cdot))$ and $\tilde{u}_t(\cdot,\tau^v(t;\cdot))$ by solving (\ref{wtilde_t})-(\ref{wtildeIC})   over determinate set $\D(t)$
\STATE solve
\begin{align*}
\tilde{v}_{\partial t}^*(x,t) =&  \partial_t U^*(t) - \int_0^x  \frac{1}{\lambda^v} \left[c_5 \,\tilde{u}_{\partial t}^*\tilde{v}_{\partial t}^* + c_6\, (\tilde{v}_{\partial t}^*)^2 \right. \\
& \quad \left. + c_7\, \tilde{u}_{\partial t}^* + c_8 \,\tilde{v}_{\partial t}^* \right] d\xi,\\
\partial_t U^*(t) =&  -\delta\times \operatorname{sign}\left(\tilde{v}(0,\tau^v(t;0))\right), 
\end{align*}
where $\tilde{u}_{\partial t}^*(\cdot,t) =\tilde{u}_t(\cdot,\tau^v(t;\cdot))$, over $x\in[0,1]$
\STATE set $\partial_t U(t) = \tilde{v}_{\partial t}^*(1,t)$
\end{algorithmic}
\caption{Continuous-time control algorithm}
\label{algorithm_continuous}
\end{algorithm}
\noindent Here, we used the same notation as before but omit the index $k$ representing sampling. To avoid chattering, it is advisable to use a smooth approximation of the $\operatorname{sign}$-function. 
\end{remark}

{
Remark \ref{remark continuous feedback} highlights one of the differences between the quasilinear systems considered here, and semilinear, including linear systems. Namely, the control law can only prescribe how  the control input changes, either by computing  the time-derivative $\partial_t U(t)$ in a continuous-time fashion as in Algorithm \ref{algorithm_continuous}, or by computing a compatible input over the interval $[t_k,t_{k+1}]$ in a sampled-time fashion as in Algorithm \ref{algorithm_statefeedback}. { Whereas in the semilinear case in $L^{\infty}$, there is no compatibility condition of form (\ref{compatibility}) and the control input can be chosen freely in $L^{\infty}$.}
%
%

Despite these differences, Algorithm \ref{algorithm_statefeedback} is  related to the controller in \cite{strecker2017output}. For semilinear systems, the state-space can be changed to $L^{\infty}$ and by choosing  $U^*(t)\equiv 0$ for all $t\geq \tau_0$ and using an infinitesimal sampling time $\theta\rightarrow 0$, Algorithm \ref{algorithm_statefeedback} yields  equivalent control inputs as the state-feedback control law from \cite{strecker2017output}.}  Moreover, for linear systems, and again in the limit $\theta\rightarrow 0$, it has been shown in \cite[Section 3.4]{strecker2017output} that the control law gives the same control inputs as the backstepping controller \cite{vazquez2011backstepping}. Conversely,  Algorithm \ref{algorithm_statefeedback} can be seen as an implementation of the sampled-control ideas indicated in \cite[Remark 6]{strecker2017output}. 

{ Algorithm \ref{algorithm_statefeedback} is also related to the open-loop control law in \cite{li2003exact}  in the sense that by choosing the sampling period $\theta \geq \tau^u\left(\tau_0+\frac{|v(0,\tau_0)|}{\delta};1\right)$ (i.e., long enough so that the state converges to the origin within the first sampling interval), it is possible to obtain the same control inputs.  If the initial condition is made small and space-invariant speeds are assumed, the convergence time in Theorem \ref{thm state feedback exact} gets arbitrarily close to the minimum convergence times stated in the local result  \cite{li2003exact}, because for $\|w_0\|\rightarrow 0$ we have $\tau_0\rightarrow \frac{1}{\lambda^v(w=0)}$, $|v(0,\tau_0)|\rightarrow 0$ and $\tau^u(\tau_0;1)\rightarrow \frac{1}{\lambda^v(w=0)}+\frac{1}{\lambda^u(w=0)}$.}

\subsection{Tracking}\label{sec:tracking}
It is straightforward to modify the control law to solve tracking problems of the form
\begin{equation}
v(0,t) = \phi(t),
\end{equation}
which includes many problems of the form (inserting (\ref{uBC}) and solving for $v(0,\cdot)$)
\begin{equation}
v(0,t) = \tilde{\phi}(u(0,t),t).
\end{equation}
Theorem \ref{thm: global existence condition}  ensures well-posedness for tracking signals $\phi$ that are bounded and have small time-derivative ($\leq \delta^{\mathrm{max}}$). The virtual input $U^*$ can be designed to converge from the prediction $v(0,\tau_k)$ to the reference signal $\phi$ with rate $\delta$, until $U^*$ becomes equal to $\phi$ within finite time. Here, $\delta\geq \|\partial_t\phi(\cdot)\|_{\infty}$ is necessary so that the control signal is able to follow the reference.

\section{A robustness certificate} \label{sec:robustness}
In this section we show that, under some more restrictive assumptions on the model coefficients, the feedback control law from Section \ref{sec: exact state feedback} has inherent robustness properties with respect to small errors in the model, control input and  measurements. For simplicity, we assume that the boundary conditions are time-invariant.

\subsection{Assumptions and notation} 
Assume  the actual, uncertain dynamics are
\begin{align}
{w}_t(x,t) &= \bar{\Lambda}(x,w(x,t))\,{w}_x(x,t) + \bar{F}(x,{w}(x,t)), \label{wtildeerror_t}\\
u(0,t)& =  \bar{g}^u(v(0,t)), \label{utildeerrorBC}\\
v(1,t)& = \bar{U}(U(t)) + \bar{g}^v({u}(1,t)), \label{vtildeerrorBC}
\end{align}
with $\bar{\Lambda}=\operatorname{diag}(-\bar{\lambda}^u,\bar{\lambda}^v)$ and $\bar{F}=(\bar{f}^u~\bar{f}^v)^T$, and where $\bar{U}$ is some nonlinear function of $U(t)$ representing actuator errors. The data $\bar{\Lambda}$, $\bar{F}$, $\bar{g}^u$ and $\bar{g}^v$   satisfy the same basic assumptions as listed in Section \ref{sec:assumptions}.  We also consider mismatch between the measurement $W_k$ and the actual state at time $t_k$. Instead of (\ref{measurement exact state feedback}),  $W_k$ is assumed to be Lipschitz-continuous and to satisfy the compatibility conditions
\begin{align}
\tilde{v}^{k}(1,t_k) &= \lim_{t\rightarrow t_k} U(t), \label{compatibility measurement}\\
\tilde{u}^{k}(0,t_k) &= g^u(\tilde{v}^{k}(0,t_k)), \label{compatibility measurement 2}
\end{align}
as well as the error bounds below.
The model, measurement and input errors are assumed to be bounded as follows:
\begin{align}
\|\bar{\Lambda}(x,z)-\Lambda(x,z)\| & \leq \epsilon_{\Lambda} \|z\|_{\infty}, \label{error bound Lambda} \\
\|\bar{F}(x,z)-F(x,z)\| & \leq \epsilon_{F} \|z\|_{\infty}, \label{error bound F}\\
\|\bar{g}^u(z,t)-g^u(z,t)\| & \leq \epsilon_{g_u} \|z\|_{\infty}, \\
\|\bar{g}^v(z,t)\| & \leq \epsilon_{g_v} \|z\|_{\infty},\\
\|w(x,t_k)-W_k(x)\| & \leq \epsilon_w \|w(\cdot,t_k)\|_{\infty}, \label{error bound w}\\
\|w_t(x,t_k)-{\footnotesize \frac{\partial}{\partial t}}W_k(x)\| & \leq \epsilon_{w_t}\|w_t(\cdot,t_k)\|_{\infty}, \label{error bound w_t}\\
|\bar{U}(U)-U| &\leq \epsilon_U |U|, \label{error bound U}
\end{align}
for some appropriate positive constants $\epsilon_{\Lambda}$, $\epsilon_{F}$ through $\epsilon_{U}$  and  all $x\in[0,1]$, $z\in\R^n$, $n\in\{1,2\}$, $k\in\N$ and $t\geq 0$. In (\ref{error bound w_t}), the derivative $\frac{\partial}{\partial t}W_k$ is obtained by substituting $W_k$ into the right-hand side of (\ref{wtilde_t}) using the nominal parameters.

{ 
Assume the coupling terms are so small that there exists a $\gamma_4>0$ with 
\begin{align}
\gamma_4 \geq \sup_{(\ref{error bound Lambda}),(\ref{error bound F})} \sup_{x\in[0,1],\,\|z\|\leq c}\max\left\{\left|\frac{\bar{c}_{6}}{\bar{\lambda}^v}\right|,\left|\frac{\bar{c}_{8}}{\bar{\lambda}^v}\right|\right\},   \label{condition 1}\\
 \sup_{\begin{subarray}{c}(\ref{error bound Lambda}),(\ref{error bound F}),\\x\in[0,1],\,\|z\|\leq c\end{subarray}} \max\{ \left|\frac{\bar{c}_{5}}{\bar{\lambda}^v}\right|,\left|\frac{\bar{c}_{7}}{\bar{\lambda}^v}\right| \} \kappa_2 \leq \exp(-4\gamma_4) \gamma_4, \label{condition 2}
\end{align}
for $c=1.5\kappa_1^2\|w_0\|_{\infty}$, $\kappa_2$ as in (\ref{kappa_2}) and where $\sup_{(\ref{error bound Lambda}),(\ref{error bound F})}$ indicates that the supremum of $\bar{c}_1$ through $\bar{c}_{8}$ is taken over all $\Lambda$ satisfying (\ref{error bound Lambda}) and all  $F$ satisfying (\ref{error bound F}), and  a $\gamma_5>0$ with
\begin{align}
\gamma_5 &\geq \sup_{x\in[0,1],\,\|z\|\leq \tilde{c}}\max\left\{\left|\frac{{c}_{6}}{\lambda^v}\right|,\left|\frac{{c}_{8}}{\lambda^v}\right|\right\},   \label{condition 11} \\
 \sup_{x\in[0,1],\,\|z\|\leq \tilde{c}} &\max\{ \left|\frac{c_{5}}{\lambda^v}\right|,\left|\frac{{c}_{7}}{\lambda^v}\right| \} \kappa_3 \kappa_2 \leq \exp(-4\gamma_{5}) \gamma_{5}, \label{condition 12}
\end{align}
for  $\tilde{c}=2\kappa_1 c$ and $\kappa_3$ as in (\ref{kappa_3}). Moreover, the uncertainties are assumed to satisfy the bounds (with $\tilde{\kappa}_{10}$ as in (\ref{hlp9}))
\begin{align}
\epsilon_w &\leq 1,& \epsilon_{w_t} &\leq 1,  \label{epsilonw<1}\\
\epsilon_{g_v}&\leq \min\left\{\frac{1}{6\kappa_2}e^{-4\gamma_4},\frac{1}{\tilde{\kappa}_{10}}\right\}, &  \epsilon_U&\leq \frac{1}{4}.\label{condition 7}
\end{align}
}

Denote the characteristic lines and determinate sets of the actual, uncertain system by
 \begin{align}
 \bar{\tau}^v(t;x) &= t+\int_x^1\frac{1}{\bar{\lambda}^v(x,{w}(x, \bar{\tau}^v(t;x)))},   \label{tau^v bar}\\
 \bar{\tau}^u(t;x) &= t-\int_x^1\frac{1}{\bar{\lambda}^u(x,{w}(x, \bar{\tau}^u(t;x)))},   \label{tau^u bar}\\
  \bar{\tau}_k &=  \bar{\tau}^v(t_k;0),   \label{tau_bar_k}\\
\bar{\D}_k &= \left\{(x,s):\,x\in[0,1],\,s\in[t_k,\bar{\tau}^v(t_k;x)]\right\}. \label{D bar}
 \end{align}
 By contrast, the characteristic lines of the prediction at the $k$-th sampling instant  with nominal parameters are denoted by
%
%
   \begin{align}
 \tilde{\tau}^v_k(t;x) &= t+\int_x^1\frac{1}{{\lambda}^v(x,\tilde{w}^{k}(x,  \tilde{\tau}^v_k(t;x)))},   \label{tau^v tilde}\\
 \tilde {\tau}^u_k(t;x) &= t-\int_x^1\frac{1}{{\lambda}^u(x,\tilde{w}^{k}(x,  \tilde{\tau}^u_k(t;x)))},   \label{tau^u tilde}\\
\tilde{\tau}_k &=   \tilde{\tau}^v_k(t_k;0),   \label{tau_tilde_k}\\
 \tilde{\D}_k &= \left\{(x,s):\,x\in[0,1],\,s\in[t_k, \tilde{\tau}^v_k(t_k;x)]\right\}. \label{D tilde}
 \end{align}
 The difference between the actual and predicted characteristic lines is illustrated in Figure \ref{fig:characteristic lines uncertain}.
 
 \begin{figure}[htbp!]\centering
\includegraphics[width=.7\columnwidth]{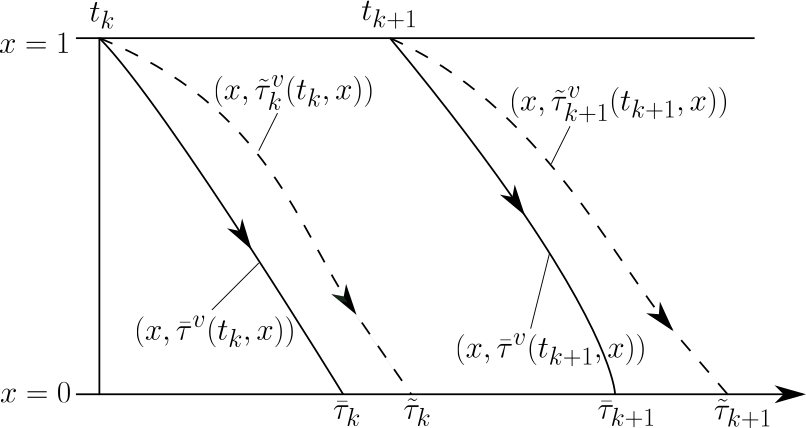}
\caption{Schematic of the characteristic lines of the actual system (solid lines) and the prediction (dashed).} 
\label{fig:characteristic lines uncertain}
\end{figure}

\subsection{Main robustness result}
We now state our main robustness result. In Section \ref{sec: exact state feedback}, we rely on being able to exactly predict and reverse the system dynamics so that the actual boundary value $v(0,\cdot)$ becomes equal to the designed one. In presence of model uncertainty, however, induced prediction errors lead to mismatch between designed and actual boundary values. We show that for sufficiently small uncertainties, the closed-loop trajectory still converges to a ball around the origin, which  can be made arbitrarily small by restricting the size of the uncertainties. The required size of the allowable uncertainties also depends on the norm of the initial condition. Moreover, during transients the solution remains bounded by a constant depending on the initial condition.

\begin{figure}[htbp!]\centering
\includegraphics[width=.9\columnwidth]{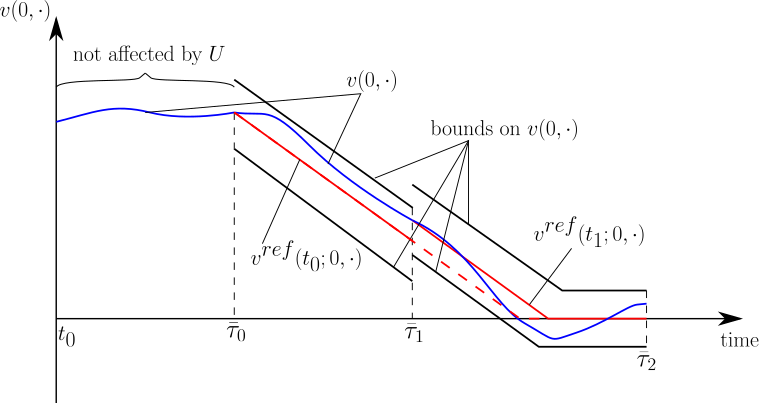}
\caption{Schematic of the trajectory of $v(0,\cdot)$ (blue) as well as $v^{\mathrm{ref}}$ (red) and bounds on $v(0,\cdot)$ (black) in each time step.}
\label{fig:proof robustness}
\end{figure}
\begin{theorem}\label{thm:robust convergence}
 Let $c=1.5\,\kappa_1^2\,\|w_0\|_{\infty}$ with $\kappa_1$ as in (\ref{kappa_1}), let $\bar{\delta}^{\mathrm{max}}$ be given by (\ref{delta bar max}) and let $\delta^{\mathrm{max}}=\delta^{\mathrm{max}}\left(\bar{\delta}^{\mathrm{max}}\right)$ be given by  (\ref{delta rob}). Assume $\|w_t(\cdot,0)\|_{\infty}\leq \frac{\bar{\delta}^{\mathrm{max}}}{\kappa_2}$ with $\kappa_2$ as in (\ref{kappa_2}).

For every $\varepsilon>0$, $\theta>0$ and $\delta\in(0,\delta^{\mathrm{max}}]$, there exist   $\epsilon^{\mathrm{max}}>0$ and $T>0$ such that if  
\begin{equation}
\epsilon_{\mathrm{sum}} = \epsilon_w+\epsilon_{\Lambda}+\epsilon_F+\epsilon_U+\epsilon_{g^u}+\epsilon_{g^v} \leq \epsilon^{\mathrm{max}}, \label{epsilon_sum}
\end{equation}
 then the uncertain closed-loop system  (\ref{wtildeerror_t})-(\ref{vtildeerrorBC}) with $U$ as computed by Algorithm \ref{algorithm_statefeedback}  has a  global solution in  $\X^{c, c^{\prime}}_{[0,1]\times[0,\infty)}$  that satisfies 
\begin{equation}
\|w(\cdot,t)\|_{\infty}  <\varepsilon \quad \text{ for all } \quad t\geq T. \label{robust convergence}
\end{equation}
\end{theorem}

The proof of Theorem \ref{thm:robust convergence} is very technical and is given in Section \ref{sec:proof robustness} following preparations over the next subsections. It consists of two steps. First, robust boundedness of the time derivative $w_t$ is shown. This ensures robust existence of the solution and is also used to bound the growth of the prediction error. Second, an estimate on the error between predicted and actual trajectories is developed to prove robust  convergence to a ball around the origin.

\subsection{Robust boundedness of $w_t$} \label{sec:robustness w_t}
In the nominal case, for any desired upper bound $\tilde{c}^{\prime}$ on $\|v_t(0,\cdot)\|_{\infty}$, choosing $\delta\leq \tilde{c}^{\prime}$ in (\ref{U^*})    automatically ensures $|v_t(0,t)|=|\partial_t U^{*,k}(t)|\leq \delta \leq \tilde{c}^{\prime}$ for all $t\geq \tau_0$. This is no longer automatically guaranteed in the presence of model uncertainty, because the actual value of $|v_t(0,t)|$ might be larger than the designed value of $|\partial_t U^{*,k}(t)|$. Instead, we seek a stronger bound on $\delta$ and conditions on the coupling coefficients between $u_t$ and $v_t$ (see (\ref{u_t integral2})-(\ref{v_t integral2})) so that $|{\partial_t } U^{*,k}(t)|\leq \delta$ still ensures $|v_t(0,t)|\leq \tilde{c}^{\prime}$.

Throughout section \ref{sec:robustness w_t}, we  fix sampling instance $k\in\N$ and suppose that the solution satisfies the a-priori bound $\sup_{t\in[0,\bar{\tau}_{k+1}]}|v(0,t)|\leq 1.5\,\kappa_1\,\|w_0\|_{\infty}$. Note that this implies $\sup_{x\in[0,1],t\in[0,\bar{\tau}^v(t_{k+1};x)]}|v(x,t)|\leq 1.5\,\kappa_1^2\,\|w_0\|_{\infty}$.

\subsubsection{Preliminaries}
With regards to the integral equations  (\ref{u_ts})-(\ref{v_ts}), let $\bar{c}_1$, $\bar{c}_2$, etc., be given by the expressions in Appendix \ref{appendix A}  with the  nominal $\lambda^v$ replaced by the uncertain $\bar{\lambda}^v$, etc.
Define 
\begin{align}
\gamma_3 &= \sup_{(\ref{error bound Lambda}),(\ref{error bound F})}\sup_{x\in[0,1],\,\|z\|\leq c}\max\left\{|\bar{c}_1+\bar{c}_2|,|\bar{c}_3+\bar{c}_4|, \right.\nonumber \\
& \qquad \qquad \qquad \left. |\bar{c}_5+\bar{c}_6|,|\bar{c}_7+\bar{c}_8| \right\},    \\
\kappa_2 &= \max\{1,l_{g^u}\} \,e^{2\gamma_3 l_{\Lambda^{-1}}}, \label{kappa_2}
\end{align}
for $c=1.5\,\kappa_1^2\,\|w_0\|_{\infty}$ and where $c_1$ etc.\ are evaluated at $(x,z)$ and  $\sup_{(\ref{error bound Lambda}),(\ref{error bound F})}$ indicates that the supremum of $\bar{c}_1$ through $\bar{c}_{8}$ is taken over all $\Lambda$ satisfying (\ref{error bound Lambda}) and all  $F$ satisfying (\ref{error bound F}).
Also define the interval
\begin{equation}
\I(t) = [\bar{\tau}^u(t;0),\bar{\tau}^v(t;0)].
\end{equation}
The following Lemma gives another expression for the relation between the state and the boundary values at $x=0$. By extending the solution to negative times, it allows a simpler characterization of $\|w_t(\cdot,t)\|_{\infty}$ in terms of $\|v_t(0,\cdot)\|_{\infty}$ alone, i.e., independently of the initial condition.
\begin{lemma} \label{lemma bijective map}
For every $t\geq 0$, there exists a bijective map 
\begin{equation}\begin{aligned}
\Phi:\,& \X_{[0,1]}^{\cdot,\cdot} \rightarrow \X_{\I(t)}^{\cdot,\cdot}\\
& w(\cdot,t) \mapsto v(0,\cdot)\label{bijective map}.
\end{aligned}\end{equation}
Moreover, with reference to (\ref{kappa_1}) and (\ref{kappa_2}),
\begin{align}
\frac{1}{\kappa_1} \|w(\cdot,t)\|_{\infty} \leq \sup_{s\in\I(t)} &|v(0,s)|  \leq \kappa_1 \|w(\cdot,t)\|_{\infty},  \label{bijective map estimates}\\
\frac{1}{\kappa_2} \|w_t(\cdot,t)\|_{\infty} \leq \operatorname*{ess\,sup}_{s\in\I(t)} &|v_t(0,s)|  \leq \kappa_2 \|w_t(\cdot,t)\|_{\infty}. \label{bijective map estimates derivative}
\end{align}
\end{lemma}
\begin{proof}
This follows directly by applying the techniques from Theorems \ref{thm:determinate set 1} and \ref{thm:determinate set 2} to system (\ref{w_t})-(\ref{uvIC}) in both the positive and negative $t$-direction, and to system (\ref{w_x})  in the positive $x$-direction, respectively. Note that $\partial_t g^u=0$ is assumed.
\end{proof}

\subsubsection{Existence of prediction}
Because of model uncertainty and the measurement error introduced at time $t_k$, it is not clear a priori that the prediction $\tilde{w}^{k}$ exists on the whole determinate set $\tilde{\D}_k$, nor that it satisfies a desirable bound on $\tilde{v}_t^{k}(0,\cdot)$, even if the actual trajectory $w$ exists on the whole of $\bar{\D}_k$ and satisfies a corresponding bound, or if the previous prediction $\tilde{w}^{k-1}$ existed on all of $\tilde{\D}_{k-1}$. It is possible to ensure existence of a solution $\tilde{w}^{k}$ on all of $\tilde{\D}_k$  by keeping $\|v_t(0,\cdot)\|_{\infty}$ sufficiently small and limiting the measurement error.
\begin{lemma} \label{lemma robust existence wtilde_t}
There exist constants $\bar{\delta}^{\mathrm{max}}_1>0$ and $\kappa_3>0$ such that if $|v_t(0,t)|\leq \bar{\delta}$ for some $\bar{\delta}\leq \bar{\delta}^{\mathrm{max}}_1$ and  all $t\in [\bar{\tau}^u(0;0),\bar{\tau}_k]$, then the prediction $\tilde{w}^{k}$ has a solution on $\tilde{\D}_k$ that satisfies
\begin{equation}
|\tilde{v}_t^{k}(0,t)|\leq \kappa_3\, \bar{\delta}  \label{vtilde_t < deltatilde}
\end{equation}
for all $t \in[t_k, \tilde{\tau}_k]$.
\end{lemma}
\begin{proof} See Appendix \ref{proof appendix 1}.
\end{proof}

\subsubsection{Bound on $|\partial_t U|$}
Similar to (\ref{v_t integral2}), $v_t$ can be integrated along its characteristic line in the \emph{negative} $x$-direction, which  gives
\begin{align}
v_t(x,\bar{\tau}^v(t;x)) = & v_t(1,t) + \int_{x}^1 \frac{1}{\bar{\lambda}^v} \left[\bar{c}_5u_t v_t + \bar{c}_6 (v_t)^2 \right. \nonumber \\
& \qquad \left. + \bar{c}_7 u_t  + \bar{c}_8 v_t   \right](\xi,\bar{\tau}^v(t;\xi))  \, d\xi \label{v_t integral3}
\end{align} 
with boundary condition
\begin{equation}
v_t(1,t) = \partial_U \bar{U}(U(t)) \,U_t(t) + \partial_u \bar{g}^v(u(1,t)) \,u_t(1,t).  \label{v_t integral3 BC}
\end{equation}
The following Lemma gives a conservative condition on the size of $|U_t|$ ensuring that $|v_t(0,t)|\leq \bar{\delta}$, based on the worst-case growth along the characteristic lines alone.

\begin{lemma}\label{lemma bound v_t 1}
Assume $|v_t(0,t)|\leq \bar{\delta}$ for some $\bar{\delta}\leq  \min\left\{\bar{\delta}^{\mathrm{max}}_1,\, \frac{\|w_0\|_{\infty}}{\theta+l_{\Lambda^{-1}}}\right\}$ and all $t\in[\bar{\tau}^u(0;0),\bar{\tau}_k]$.
If
\begin{equation}
|\partial_t U(t)| \leq \frac{2}{3}\,e^{-4\gamma_4}\,\bar{\delta}\label{condition 8}
\end{equation}
for all $t\in[t_k,t_{k+1}]$, then 
\begin{equation}
|v_t(0,t)|\leq \bar{\delta}
\end{equation}
for all $t\in[\bar{\tau}_k,\bar{\tau}_{k+1}]$.
\end{lemma}
\begin{proof} See Appendix \ref{proof appendix 2}.
\end{proof}

\subsubsection{Bound on $\delta$}

Again based on a worst-case estimate of the growth of $\tilde{v}_t^{k}(x,t)$ along its characteristic lines, one can ensure that (\ref{condition 8}) holds  by designing $|\partial_t U^*(t_k;\cdot)|$ to be sufficiently small. 

\begin{lemma}\label{lemma bound v_t 2}
 Choosing $\delta$ such that
\begin{equation}
\delta \leq \delta^{\mathrm{max}}\left(\bar{\delta}\right) =  \frac{2}{3}e^{-4(\gamma_4+\gamma_5)}\bar{\delta}   \label{delta rob}
\end{equation}
ensures that $U$ as constructed by Algorithm \ref{algorithm_statefeedback} satisfies (\ref{condition 8}).
\end{lemma}
\begin{proof} 
The proof is similar to the proof of Lemma \ref{lemma bound v_t 1}  and uses (\ref{condition 11})-(\ref{condition 12}). 
\end{proof}

\subsection{Error bounds} \label{sec: error bounds}
In the following we derive  bounds on the state prediction error, the error between inputs computed by use of nominal and actual parameters, and the resulting error in the boundary value $v(0,\cdot)$. Throughout section \ref{sec: error bounds}, we  again fix  the sampling instance $k\in\N$ and  assume  the  a-priori bound $\sup_{t\in[0,\bar{\tau}_{k+1}]}|v(0,t)|\leq 1.5\,\kappa_1\,\|w_0\|_{\infty}$. 

\subsubsection{Prediction error}

\begin{lemma}\label{lemma: prediction error}
Choose $\delta$ as in $(\ref{delta rob}$) for some $\bar{\delta} \leq \bar{\delta}^{\mathrm{max}}$ with
\begin{equation}
\bar{\delta}^{\mathrm{max}} =  \frac{1}{2}\min\left\{\frac{1}{\kappa_2\,l_{\Lambda^{-1}}^2\, l_{\Lambda}},\,\frac{e^{-2\tilde{\gamma}_1 l_{\Lambda^{-1}}}}{\kappa_2\,\max\{1,l_{g^u}\}},\, \frac{\|w_0\|_{\infty}}{\theta+l_{\Lambda^{-1}}}\right\}. \label{delta bar max}
\end{equation}
There exists  a constant $\kappa_4>0$  such that for all $x\in[0,1]$,
\begin{equation}\begin{aligned}
\|w(x,\bar{\tau}^v(t_k;x))&-\tilde{w}^{k}(x,\tilde{\tau}_k^v(t_k;x))\|   \\
 \leq \kappa_4&\,(\epsilon_w+\epsilon_{\Lambda}+\epsilon_F+\epsilon_{g^u})\, \|w(\cdot,t_k)\|_{\infty}. \label{bound prediction error}
\end{aligned}\end{equation}
\end{lemma}
\begin{proof} See Appendix \ref{proof appendix 3}.
\end{proof}
\subsubsection{Error in control input}
Next, we derive a bound on the error between the computed control input and the control input that would be obtained by use of the real parameters.

Let $w^{\mathrm{ref}, k}=\left(\begin{matrix}u^{\mathrm{ref}, k} & v^{\mathrm{ref}, k} \end{matrix} \right)^T$ denote the trajectories that are obtained by use of control law as in Algorithm \ref{algorithm_statefeedback} but with the actual, unknown parameters ($\bar{\Lambda}$, $\bar{F}$, etc.) instead of the nominal ones, and let $U^{*,ref}$ be the corresponding virtual input. 
\begin{lemma}\label{lemma: input error}
 Choose $\delta$  as in Lemma \ref{lemma: prediction error}. There exists  $\kappa_5>0$ such that  the error between the reference boundary value $v^{\mathrm{ref}, k}(1,\cdot)$ and the actual boundary value $v(1,\cdot)$ is bounded by 
\begin{equation} \begin{aligned}
|v(&1,t)-v^{\mathrm{ref}, k}(1,t)|\leq \kappa_5 \,\epsilon_{\mathrm{sum}} \, \|w(\cdot,t_k)\|_{\infty}  \label{bound input error}\\
\end{aligned}\end{equation}
for all $t\in[t_k,t_{k+1}]$,  with $\epsilon_{\mathrm{sum}} $ as defined in (\ref{epsilon_sum}).
\end{lemma}
\begin{proof} See Appendix \ref{proof appendix 4}.
\end{proof}

\subsubsection{Error in boundary value $v(0,\cdot)$}

We can now bound the error  between the boundary value $v(0,\cdot)$ and the boundary value of the reference trajectory, $v^{\mathrm{ref}, k}(0,\cdot)$, that would be obtained if the actual parameters were known.
\begin{lemma}  \label{lemma error v(0)}
 Choose $\delta$ as in Lemma \ref{lemma: prediction error}.  There exists a constant $\kappa_6>0$ such that 
\begin{equation} \begin{aligned}
|v(0,\bar{\tau}^v(t;0))-v^{\mathrm{ref}, k}&(0,\bar{\tau}^v(t;0))| \\& \leq \kappa_6\,\epsilon_{\mathrm{sum}}\, \|w(\cdot,t_k)\|_{\infty} \label{bound boundary error}
\end{aligned}\end{equation}
for all $t\in[t_k,t_{k+1}]$.
\end{lemma}
\begin{proof}
The proof is very similar to the derivation of (\ref{bound input error preliminary})  and uses (\ref{bound input error}).
\end{proof}

\begin{remark}\label{remark kappa6}
Close inspection of the proofs of Lemmas \ref{lemma: prediction error}-\ref{lemma error v(0)} reveals that the constant $\kappa_6$ scales exponentially with $(\tau_{k+1}-t_k)$, i.e.,~the sum of the delay $\tau_{k+1}-t_{k+1}$ and the sampling interval $\theta$. That is, shorter sampling intervals reduce the prediction error.  This raises the conjecture that if actuators are available at both boundaries, bilateral control as in \cite{strecker2017twosided} might lead to smaller prediction errors and, thus, better robustness margins due to the shorter prediction horizons.
\end{remark}

\subsection{Proof of Theorem \ref{thm:robust convergence}}\label{sec:proof robustness}
We are now in position to prove convergence of the uncertain closed-loop system to a ball around the origin.

\begin{proof}[Proof of Theorem \ref{thm:robust convergence}]
The idea of the proof is to show that 
\begin{equation}
|v(0,t)|\leq \frac{1}{\kappa_1} \varepsilon  \label{proof rob stab 4}
\end{equation}
for all $t\geq T^{\prime}$ for some $T^{\prime}>0$. By virtue of  Lemma \ref{lemma bijective map}, this ensures that $\|w(\cdot,t)\|_{\infty} \leq \varepsilon$ for all $t\geq T=T^{\prime}+l_{\Lambda^{-1}}$.
 The trajectory of $v(0,\cdot)$ and the main  ideas behind the proof are also sketched in Figure \ref{fig:proof robustness}.

We will show the bounds
\begin{align}
|v(0,\bar{\tau}_k)|&=|v^{\mathrm{ref}, k}(0,\bar{\tau}_{k})|\leq \kappa_1 \|w_0\|_{\infty}, \label{proof rob stab 6}\\
\sup_{t\in[0,\bar{\tau}_k]}|v(0,t)|& \leq 1.5\kappa_1 \|w_0\|_{\infty}, \label{proof rob stab 61}
\end{align}
for all $k\in\N$ by induction. For $k=0$, (\ref{proof rob stab 6})-(\ref{proof rob stab 61}) hold due to Theorem \ref{thm:determinate set 1}. Assume they hold up to $k$ and we show  that they also hold for $k+1$. 

Lemmas \ref{lemma bound v_t 1} - \ref{lemma bound v_t 2} and the design of $U^*$ with $\delta\leq \delta^{\mathrm{max}}$, and the assumption  $\|w_t(\cdot,0)\|_{\infty} \leq \frac{\bar{\delta}^{\mathrm{max}}}{\kappa_2}$, ensure that  $|v_t(0,t)|\leq \bar{\delta}^{\mathrm{max}}$ as long as $\sup_{s\in[0,t]}|v(0,s)| \leq 1.5\kappa_1 \|w_0\|_{\infty}$. Since $\tau_{k+1}-\tau_k\leq \tau_{k+1}-t_k\leq  \theta+l_{\Lambda^{-1}}$ and since $\bar{\delta}^{\mathrm{max}} \leq \frac{\|w_0\|_{\infty}}{\theta+l_{\Lambda^{-1}}}$, (\ref{proof rob stab 6}) implies that (\ref{proof rob stab 61}) also holds for $k+1$. 
Therefore, $\sup_{t\leq \bar{\tau}_{k+1}}|v_t(0,t)|\leq \bar{\delta}^{\mathrm{max}}$  and $\sup_{x\in[0,1],t\in[0,\bar{\tau}^v(t_{k+1};x)]}{\|}w_t(x,t)\|\leq \kappa_2 \bar{\delta}^{\mathrm{max}}$.

We next search for a lower bound $\sigma>0$ such that 
\begin{equation}
\bar{\tau}_{k}+\sigma \leq \bar{\tau}_{k+1} \label{rho_1}
\end{equation} 
is guaranteed. Let $\tilde{t}(\sigma)$ be defined such that $\bar{\tau}^v(\tilde{t};0) = \bar{\tau}_k+\sigma$. Consider the difference $\Delta(x) = \bar{\tau}^v(\tilde{t};x)-\bar{\tau}^v(t_k;x)$.
We have $\Delta(0) = \sigma$ and (applying the chain rule)
\begin{equation}\begin{aligned}
|\partial_x \Delta(x)|  & =   \left|\frac{1}{\bar{\lambda}^v\left(x,w(x,\bar{\tau}^v(t_k;x))\right)}\right. \\
& \qquad -\left.  \frac{1}{\bar{\lambda}^v\left(x,w(x,\bar{\tau}^v(\tilde{t};x))\right)}\right| \\
 &\leq { \|\partial_w\left(\bar{\lambda}^v\right)^{-1}\|_{\infty}\times \|\partial_t w\|_{\infty} \times \Delta(x)} \\
& \leq  ( l_{\Lambda^{-1}}^2 l_{\Lambda}  )\times ( \kappa_2 \bar{\delta}^{\mathrm{max}} )\times  \Delta(x).
\end{aligned}
\end{equation}
Thus, $\tilde{t}(\sigma) - t_k = \Delta(1)\leq e^{l_{\Lambda^{-1}}^2 l_{\Lambda} \kappa_2 \bar{\delta}^{\mathrm{max}}}\sigma $ and, consequently, (\ref{rho_1}) can be ensured for $\sigma=  e^{-l_{\Lambda^{-1}}^2 l_{\Lambda} \kappa_2 \bar{\delta}^{\mathrm{max}}}\theta$.

Using (\ref{bound boundary error}) and (\ref{bijective map estimates}), we get for $t\in[t_k,t_{k+1}]$, $k\in\N$,
\begin{equation}\begin{aligned}
|&v(0,\bar{\tau}^v(t;0))| \leq  |v^{\mathrm{ref}, k}(0,\bar{\tau}^v(t;0))| \\
&+ |(v(0,\bar{\tau}^v(t;0))-v^{\mathrm{ref}, k}(0,\bar{\tau}^v(t;0)))|\\
& \leq |v^{\mathrm{ref}, k}(0,\bar{\tau}^v(t;0))| + \kappa_1\kappa_6 \epsilon_{\mathrm{sum}} \max_{s\in\I(t_k)} |v(0,s)| . \label{proof rob stab}
\end{aligned}\end{equation}

We then have the rough bound
\begin{equation}\begin{aligned}
\max_{s\in\I(t_k)} |v(0,s)| &\leq |v(0,\bar{\tau}_k)| + 2 l_{\Lambda^{-1}} \bar{\delta}^{\mathrm{max}} \\
& \leq \kappa_1 \|w_0\|_{\infty} + 2 l_{\Lambda^{-1}} \bar{\delta}^{\mathrm{max}}  \label{proof rob stab 2}.
\end{aligned}\end{equation}
Let
\begin{equation}
\epsilon^{\mathrm{max}}= \frac{\min\{\delta\sigma,\,\frac{1}{\kappa_1}\varepsilon,\,\kappa_1\|w_0\|_{\infty}\}}{2\kappa_1\kappa_6(\kappa_1 \|w_0\|_{\infty} + 2 l_{\Lambda^{-1}} \bar{\delta}^{\mathrm{max}})}. \label{eq: epsilonmax}
\end{equation} 
Bounding the last term on the right-hand side of  (\ref{proof rob stab}) by (\ref{proof rob stab 2}) and with $\epsilon_{\mathrm{sum}}\leq \epsilon^{\mathrm{max}}$,  using $v^{\mathrm{ref}, k}(0,\cdot)=U^{*,k}(\cdot)$ and $\bar{\tau}_{k+1}-\bar{\tau}_{k}\geq \sigma$,  we obtain from (\ref{proof rob stab}) that
\begin{equation}\begin{aligned}
|v(0,\bar{\tau}_{k+1})| \leq &\,\max\{|v(0,\bar{\tau}_{k})| -  \sigma \delta,\,0\} \\ &~+ \frac{1}{2}\min\{\delta\sigma,\,\frac{\varepsilon}{\kappa_1},\,\kappa_1\|w_0\|_{\infty}\}. \label{proof rob stab 5 }
\end{aligned}\end{equation}
That is, the sequence $|v(0,\bar{\tau}_{k})|$, $k\in\N$, is strictly decreasing (by at least $\frac{\sigma \delta}{2}$ per time step) until it satisfies  both $|v(0,\bar{\tau}_{k})|\leq \frac{\varepsilon}{2\kappa_1}$ and $|v(0,\bar{\tau}_{k})|\leq \frac{\kappa_1\|w_0\|_{\infty}}{2}$ for all $k\geq K=\left\lceil \frac{2\kappa_1\|w_0\|_{\infty}}{\sigma \delta}\right\rceil$.  This completes the induction step for (\ref{proof rob stab 6}). Moreover, once $|v(0,\bar{\tau}_{k})|\leq \frac{\varepsilon}{2\kappa_1}$, the boundary value during the interval $s\in[\bar{\tau}_{k},\bar{\tau}_{k+1}]$ satisfies
\begin{equation}\begin{aligned}
|v(0,s)|  & \leq |v(0,\bar{\tau}_{k})| + \frac{1}{2}\min\{\delta\sigma,\,\frac{\varepsilon}{\kappa_1},\,\kappa_1\|w_0\|_{\infty}\}\\
&\leq  \frac{\varepsilon}{\kappa_1}. \label{proof rob stab 8}
\end{aligned}\end{equation}
That is, (\ref{proof rob stab 4}) holds for all $t\geq \bar{\tau}_K$.

Finally, we show the solution lies in $\X^{c, c^{\prime}}_{[0,1]\times[0,\infty)}$. Compatibility condition (\ref{compatibility measurement}) ensures that the input $U(t)$ is continuous at times $t=t_k$, $k\in\N$, and  Lipschitz-continuity of $U$ during each interval $[t_k,t_{k+1}]$ follows from the assumption that the measurement is Lipschitz, satisfies (\ref{compatibility measurement})-(\ref{compatibility measurement 2}), and the construction in Algorithm \ref{algorithm_statefeedback}.   The bound  $\|w\|_{\infty}\leq c$ follows from  (\ref{proof rob stab 61}) and Theorem \ref{thm: global existence condition}.
\end{proof}

As the sampling time $\theta$ is increased from zero, the maximum allowable error $\epsilon^{\mathrm{max}}$ in (\ref{eq: epsilonmax}) initially increases approximately linearly with $\theta$, since larger $\theta$ increases $\sigma$, thus increasing $\epsilon^{\mathrm{max}}$ in (\ref{eq: epsilonmax}). However, for larger values of $\theta$, $\epsilon^{\mathrm{max}}$ decreases approximately exponentially with $\theta$ because of the exponential dependence of $\kappa_6$ in (\ref{eq: epsilonmax}) on $\theta$; see also Remark \ref{remark kappa6}. 

The exponential decay for larger $\theta$ is intuitive because in order to limit the prediction errors, longer prediction horizons should allow less uncertainty. The reduction in robustness as $\theta$ approaches zero can be attributed to the fact that on each sampling interval, the target boundary value $v^{\mathrm{ref}, k}(0,\cdot)$ is reduced by $\delta\times(\tau_{k+1}-\tau_k)$. Therefore, shorter sampling intervals mean that the state has less time to converge towards the origin before a new error is introduced. Conversely, this can be seen as a type of dwell-time constraint \cite{hespanha1999dwelltime}, where for a given degree of uncertainty the sampling frequency must be bounded. 

 Similarly, larger $\delta$ improves the decay rate of the reference trajectory and improves the robustness margin in (\ref{eq: epsilonmax}), but  $\delta$ must also be sufficiently small to prevent blow-up of the gradient.


\section{Implementation and numerical example}\label{sec:numerical example}
\begin{figure}[htbp!]
\includegraphics[width=.48\columnwidth]{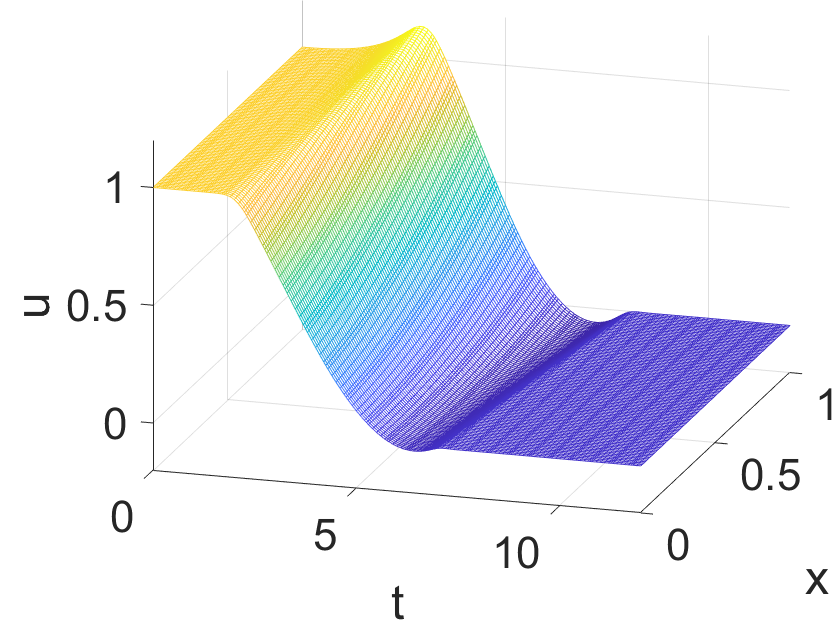} \hfill
\includegraphics[width=.48\columnwidth]{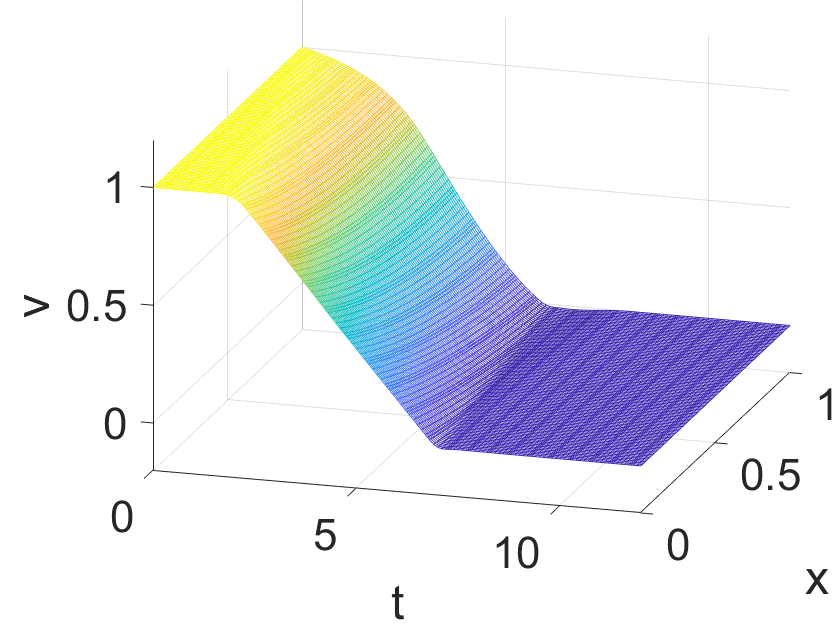}
\caption{Closed-loop trajectories.}
\label{fig:sim1}
\end{figure}

We demonstrate the controller performance in a numerical example with
$\lambda^u=1$, $\lambda^v= \max\{1-0.5\,|v|,0.2\}$, $f^u=-f^v= \frac{2}{3}(u-v)$, $  g^u = 1-\cos(2v|_{x=0})+v|_{x=0}\cos(2)$, $u_0=v_0=1$, $ \theta = 0.25$ and $ \delta=0.2$. In open loop, both the origin and the initial condition are unstable equilibria.

The closed-loop trajectories are depicted in Figure \ref{fig:sim1}. The evolution of  $U(t)$, $v(0,t)$ and $\|w(\cdot,t)\|_{\infty}$ is also shown in Figure \ref{fig:sim2}. Once the effect of the control input reaches the uncontrolled boundary, $v(0,t)$ decays linearly with rate $\delta$ according to theory. 

For the simulations, all PDEs (i.e., system (\ref{w_t})–(\ref{uvIC}), the prediction (\ref{wtilde_t})-(\ref{wtildeIC})\footnote{Rather than solving (\ref{wtilde_t})-(\ref{wtildeIC}) over the irregularly shaped domain $\D(t_k)$, it is more convenient to implement the prediction by solving (\ref{wtilde_t})-(\ref{wtildeIC}) over the  rectangular domain $[0,1]\times[t_k,t_{end}]$ with  sufficiently large $t_{end}\geq \tau_{k}$ and arbitrary, compatible value for $U$ (e.g.,~$U(t) \equiv v(1,t_k)$), and then selecting the required part of the solution from the longer prediction. Recall that, as established in Theorem \ref{thm:determinate set 1}, $U$ does not affect the solution on the domain $\D(t_k)$.} and target dynamics (\ref{target 1})-(\ref{target last})) are discretized in space using first-order finite differences, leading to a high-order ODE (‘‘method of lines’’) that is solved by use of Matlab’s ode45. For a spatial grid with 100 elements, which has been used to produce the figures, the average computation time to evaluate the output feedback controller is about 0.1\,s on a standard laptop (i.e., less than half of $\theta$ if the time units are assumed to be in seconds), although it should be emphasized that the current code has not been optimized for performance.

\begin{figure}[htbp!]\centering
\includegraphics[width=.95\columnwidth]{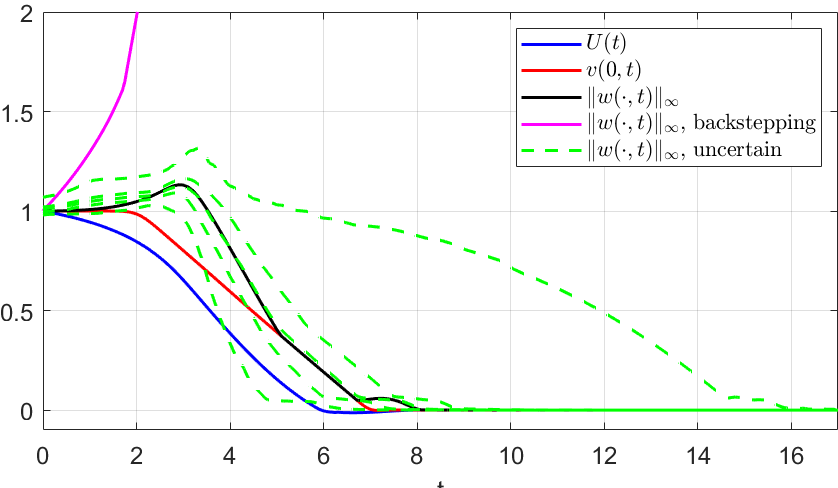} 
\caption{Control input $U(t)$, boundary value $v(0,t)$ and norm $\|w(\cdot,t)_{\infty}\|$ of nominal simulation, norm using backstepping controller and distribution of the norms under uncertainty (1st, 25th, 50th, 75th and 99th percentiles). }
\label{fig:sim2}
\end{figure}

Figure \ref{fig:sim2} also shows a comparison with the backstepping controller from \cite{vazquez2011backstepping} {, which has been computed by linearizing the system at the origin. When this linear controller is applied to the nominal nonlinear model,  the states diverge quickly.}

Finally, the sensitivity to multiplicative uncertainties of the form $\bar{f}^{u}=(1+\varepsilon_{f^{u}})\times f^{u}$, 
$\tilde{v}^k(\cdot,t_k)=(1+\varepsilon_{v})\times v(\cdot,t_k)$ with $v_0=\frac{1}{1+\varepsilon_{v}}$,\footnote{A piecewise affine function is added to the erroneous measurement to ensure the compatibility conditions (\ref{compatibility measurement})-(\ref{compatibility measurement 2}).} $\bar{U}=(1+\varepsilon_U)\times U$, etc., is investigated for $\epsilon_F=10\%$, $\epsilon_{\Lambda}=\epsilon_{g^u}=4\%$, and $\epsilon_w=\epsilon_U=\epsilon_{g^v}=2\%$. The various percentiles of the norms at each time have been computed from 1024 simulations (all 512 extremes, $\varepsilon_{f^{u}}=\pm \epsilon_F$, etc., and 512 random samples). As can be seen in Figure \ref{fig:sim2}, the trajectories under uncertainty not only remain stable but actually all converge to the origin, although the convergence speed varies between the samples.

\section{Conclusion}\label{sec conclusion}
A sampled-time predictive feedback controller is synthesized for quasilinear hyperbolic systems with only one boundary control input.  The contribution can be seen as a closing of the gap between existing feedback control methods for linear and semilinear methods \cite{vazquez2011backstepping,strecker2017output} and open-loop controllability results for quasilinear systems \cite{li2003exact,cirina1969boundary}. A robustness certificate shows that the method is inherently robust to small errors in model parameters, actuation and measurements.

The choice of $\delta$ or, more generally, the design of $U^*$ leaves degrees of freedom for tuning or even optimization of transients. The associated trade-offs remain to be explored. 
In future work, it would be desirable to obtain sharper, less conservative  conditions than the estimates of the maximal allowable decay rate $\delta^{\mathrm{max}}$ and uncertainty $\varepsilon^{\mathrm{max}}$ provided. 

 A state estimation scheme for the type of system considered in this paper is presented in \cite{li2008observability}. By evaluating this observer at every sampling instance and combining it with the state-feedback controller presented here, one solves the corresponding output feedback control problem. Alternatively, the observer from \cite{strecker2017output} can be extended to quasilinear systems.  

The developments presented in this paper are amenable to several variations. 
For instance, classical $C^1$ solutions can be obtained if the initial condition, control input and coefficients are $C^1$-functions and satisfy the $C^1$-compatibility conditions, and the virtual input is modified to satisfy the additional condition $\partial_t U^{*,k}(\tau_k)=\tilde{v}_t^k(0,\tau_k)$.

If the Lipschitz constants of $\Lambda$, $F$, $g^u$ and $g^v$, or $\Lambda^{-1}$,  are only locally bounded, one can obtain  a local result where smallness of the initial condition is required. Similarly, if  (\ref{bound Lambda_w})-(\ref{bound Lambda-1}) only hold on some (not necessarily bounded) subset of the state-space, at least in some cases it is possbile to  ensure that under feedback control the states remain in this subset.   See also \cite{gugat2003global} for such a result for so-called subcritical states of the Saint-Venant equations.

For  time-varying $\Lambda$ and $F$, the integrands in  (\ref{u_ts})-(\ref{v_ts})  have  additional  terms with $\Lambda_t$ and $F_t$ that are not multiplied with the state. If it were possible to find a result similar to Lemma \ref{lemma bound quadratic} for such integral equations, then the methodology in this paper would directly carry over to systems with slowly varying $\Lambda$ and $F$.

For clarity of presentation, we restricted the system to scalar-valued $u(x,t)$. However, the results directly carry over to systems with one actuated state and an arbitrary number of states convecting in the opposite direction (that is, to vector-valued $u(x,t)\in\R^n$ for arbitrary $n\in\N$, as in the linear case considered in \cite{dimeglio2013stabilization}). The prediction operators on $\D(t_k)$ and $\A(t_k)$ are still well-posed and straightforward to implement. The corresponding convergence time in Theorem \ref{thm state feedback exact}, i.e., $\tau^u$ in Equation (\ref{eq exact state feedback thm}), would need to be replaced by the slowest characteristic line of $u$. 

More types of uncertainty, such as noise, disturbances and input delays could be considered in future work. For input delays, it has been shown that cancelling the boundary reflection term $g^v$ can lead to robustness issues \cite{auriol2018delayrobust}, which might necessitate changes to the control design.

Finally, given the  array of systems for which controllability results exist, the method should be extended to other classes of interest such as systems with more than one control input and networks of hyperbolic systems.

\appendices
\section{Coefficients for integral equations} \label{appendix A}
Define $(\xi^u_0(x,t),t^u_0(x,t))$ as the intersection of $(\xi^u(x,t;s),s)$ with either $[0,1]\times 0$ or $0 \times [0,\infty)$, and $(\xi^v_0(x,t),t^v_0(x,t))$ as the intersection of $(\xi^v(x,t;s),s)$ with either $[0,1]\times 0$ or $1 \times [0,\infty)$.

The initial and boundary values in (\ref{u_ts})-(\ref{v_ts}) are
\begin{align}
u_t^0(x,t) &=  \begin{cases} u_t(\xi^u_0(x,t),0) & \text{ if } t^u_0(x,t)=0  \\ 
u_t(0,t^u_0(x,t)) & \text{ if } \xi^u_0(x,t)=0  \end{cases},  \label{u_t0} \\
v_t^0(x,t) &= \begin{cases} v_t(\xi^v_0(x,t),0) & \text{ if }  t^v_0(x,t)=0  \\  U_t(t^v_0(x,t)) &  \text{ if } \xi^v_0(x,t)=1 \end{cases}. \label{v_t0}
\end{align}
The coefficient functions are (omitting the arguments for brevity)
\begin{align}
c_1&=\frac{\partial_u\lambda^u}{\lambda^u}, & c_2&= \frac{\partial_v\lambda^u}{\lambda^u},   \\
c_3&= \partial_uf^u- \frac{\partial_u\lambda^u}{\lambda^u}f^u,  & c_4&= \partial_vf^u- \frac{\partial_v\lambda^u}{\lambda^u}f^u,  \\
c_5&=\frac{\partial_u\lambda^v}{\lambda^v}, & c_6&= \frac{\partial_v\lambda^v}{\lambda^v},   \\
c_7&= \partial_uf^v- \frac{\partial_u\lambda^v}{\lambda^v}f^v,  & c_8&= \partial_vf^v- \frac{\partial_v\lambda^v}{\lambda^v}f^v.
\end{align}

\section{Technical proofs for Section \ref{sec:robustness}}
\subsection{Proof of Lemma \ref{lemma robust existence wtilde_t}}\label{proof appendix 1}
The bound $\epsilon_{w}\leq 1$  in (\ref{epsilonw<1}) implies $\|\tilde{w}^{k}(\cdot,t_k)\|_{\infty}\leq 2c$,  hence the a-priori bound $\sup_{(x,t)\in\tilde{\D}_k}\|\tilde{w}^{k}(x,t)\|_{\infty}\leq 2c\kappa_1$. Let $\tilde{\gamma}_1$ be defined as in (\ref{gamma_1}) but with the supremum taken over $2\kappa_1c$. As in Theorem \ref{thm:determinate set 1} (note that $\partial_t g^u=0$ is assumed), 
\begin{equation}
\|\tilde{w}_t^{k}(\cdot,t_k)\|_{\infty}\leq \frac{1}{\max\{1,l_{g^u}\}}e^{-2\tilde{\gamma}_1 l_{\Lambda^{-1}}} \label{eq hlp_0}
\end{equation}
is sufficient to ensure existence of  $\tilde{w}^{k}$ on $\tilde{\D}_k$. Because of (\ref{error bound w_t}) and  $\epsilon_{w_t}\leq 1$  in (\ref{epsilonw<1}), (\ref{eq hlp_0}) is  ensured by  
\begin{equation}
\|{w}_t(\cdot,t_k)\|_{\infty}\leq \frac{1}{2\,\max\{1,l_{g^u}\}}e^{-2\tilde{\gamma}_1 l_{\Lambda^{-1}}}. \label{eq hlp_00}
\end{equation}
By virtue of Lemma \ref{lemma bijective map}, (\ref{eq hlp_00}) is ensured by
\begin{equation}
\bar{\delta}^{\mathrm{max}}_1=\frac{1}{2\,\kappa_2\,\max\{1,l_{g^u}\}}e^{-2\tilde{\gamma}_1 l_{\Lambda^{-1}}}. \label{delta_bar_max_1}
\end{equation}
A conservative estimate of $\kappa_3$ is 
\begin{equation}
\kappa_3=2\kappa_2\max\{1,l_{g^u}\}e^{2\tilde{\gamma}_1 l_{\Lambda^{-1}}}. \label{kappa_3}
\end{equation}

\subsection{Proof of Lemma \ref{lemma bound v_t 1}}\label{proof appendix 2}
The proof of Lemma \ref{lemma bound v_t 1} makes use of the following ancillary Lemma, which  gives an estimate of $|u_t|$ on some particular determinate sets.
\begin{lemma}
For every $t> 0$, 
\begin{equation}
\operatorname*{ess\,sup}_{\begin{subarray}{c}x\in[0,1]\\s \in[\bar{\tau}^u(0;x), \bar{\tau}^v(t;x)]\end{subarray}} |u_t(x,s)| \leq \kappa_2 \operatorname*{ess\,sup}_{s \in [\bar{\tau}^u(0;0), \bar{\tau}^v(t;0))} |v_t(0,s)|. \label{condition 3}
\end{equation}
\end{lemma}
\begin{proof}
For any $t^{\prime}<t$ the inequality  
\begin{equation}\begin{aligned}
\operatorname*{ess\,sup}_{\begin{subarray}{c}x\in[0,1],\,s \in[\bar{\tau}^u(0;x), \bar{\tau}^v(t^{\prime};x)]\end{subarray}}& \|w_t(x,s)\| \\\leq \kappa_2 &\operatorname*{ess\,sup}_{s \in [\bar{\tau}^u(0;0), \bar{\tau}^v(t^{\prime};0)]} |v_t(0,s)|
\end{aligned}\end{equation}
 follows by the same derivation as (\ref{bijective map estimates derivative}), i.e., the domain below the $\sup$ on the left-hand side is a determinate set for boundary values given on the domain  below the $\sup$ on the right-hand side when the system is solved starting from $x=0$ in the positive $x$-direction (note the closed interval  below the $\sup$ on the right-hand side as opposed to the half-open interval in (\ref{condition 3})). 
For points with $s= \bar{\tau}^v(t;x)$, the inequality still holds for $u_t$ (but not for $v_t$) because for such points, the whole integration path in (\ref{u_t integral2}) except for one point (a set of measure zero which does not affect the integral) lies in a set satisfying $s\leq\bar{\tau}^v(t^{\prime};x)$ for some $t^{\prime}<t$.
\end{proof}
Let $\varphi$ be the (positive) solution of 
\begin{align}
\varphi(x) = & \exp(-4\gamma_4)\bar{\delta} + \int_{x}^1  2\gamma_4 \left(\varphi(\xi) + (\varphi(\xi))^2 \right)  \, d\xi \label{blabla}
\end{align}
on $x\in[0,1]$. By Lemma \ref{lemma bound quadratic}, $\varphi(0)\leq \bar{\delta}$. 

Using condition (\ref{condition 2}) and (\ref{condition 3}), the terms involving $u_t$ in (\ref{v_t integral3}) can be bounded by 
\begin{align}
\left|\frac{\bar{c}_{5}}{\bar{\lambda}^v} u_t(x,\bar{\tau}^v(t;x))\right|&\leq \left|\frac{\bar{c}_{5}}{\bar{\lambda}^v} \right| \bar{\delta} \kappa_2 \leq \bar{\delta}  \exp(-4\gamma_4) \gamma_4 \nonumber \\
&    = \gamma_4 \varphi(1) \leq \gamma_4 \varphi(x), \label{condition 4}\\
\left|\frac{\bar{c}_{7}}{\bar{\lambda}^v}u_t(x,\bar{\tau}^v(t;x))\right| &\leq\gamma_4 \varphi(x). \label{condition 5}
\end{align}
If 
\begin{equation}
|v_t(1,t)|\leq \exp(-4\gamma_4)\bar{\delta}  =\varphi(1), \label{condition 6}
\end{equation}
  by use of (\ref{condition 1}) and (\ref{condition 4})-(\ref{condition 5}), and using the fact that the solution of $\phi(x)=\phi_1 + \int_x^1 \phi(\xi)\varphi(\xi) + \varphi(\xi)^2+\varphi(\xi)+\phi(\xi)d\xi$, $x\in[0,1]$, is strictly increasing in $\phi_1$, we can bound  (\ref{v_t integral3}) as
\begin{equation}\begin{aligned}
|&v_t(0,\bar{\tau}^v(t;0))| \leq  \varphi(1) + \int_{0}^1  \gamma_4 \left[ \varphi(\xi) |v_t(\xi,\tau^v(\xi))| \right.\\
& \quad \left.+ |v_t(\xi,\tau^v(\xi))|^2 + \varphi(\xi) + |v_t(\xi,\tau^v(\xi))| \right]  d\xi \\
&\leq  \varphi(1)  + \int_{0}^1  2\gamma_4 \left(\varphi(\xi) + (\varphi(\xi))^2 \right)  \, d\xi = \varphi(0) \leq \bar{\delta}.
\end{aligned}
\end{equation}
In view of (\ref{v_t integral3 BC}), (\ref{condition 3}), $|\partial_U \bar{U}|\leq 1+\epsilon_U$ and $|\partial_u\bar{g}^v|\leq \epsilon_{g^v}$, (\ref{condition 6}) is ensured by imposing conditions (\ref{condition 7})  and (\ref{condition 8}).

\subsection{Proof of Lemma \ref{lemma: prediction error}}\label{proof appendix 3}
The proof of  Lemma \ref{lemma: prediction error} consists of three steps: bounding the prediction error on common determinate set of predicted and real trajectory; bounding the difference between predicted and actual characteristic line; combining the two to obtain (\ref{bound prediction error}).

Denote the prediction error as
\begin{equation}
e^{w,k}(x,t)=\left(\begin{matrix}e^{u,k}(x,t)\\e^{v,k}(x,t)\end{matrix}\right)=\tilde{w}^{k}(x,t)-w(x,t).
\end{equation}
Subtracting (\ref{wtildeerror_t})-(\ref{vtildeerrorBC})  from (\ref{wtilde_t})-(\ref{vtildeBC}), substituting $\tilde{w}=w+e^w$ in the right-hand side and expanding, the error dynamics are
\begin{align}
&e^{w,k}_t(x,t)= {\Lambda}(x,\tilde{w}^{k})e^w_x  \nonumber\\
& + \left[\left({\Lambda}(x,e^{w,k}+w)-{\Lambda}(x,w)\right)+ \left({\Lambda}(x,w)-\bar{\Lambda}(x,{w}) \right)\right]w_x\nonumber\\
& +  \left[\left({F}(x,e^{w,k}+w)-{F}(x,w)\right)+ \left({F}(x,w)-\bar{F}(x,{w}) \right)\right],
\label{e^w 2}\\
&e^{u,k}(0,t) =  \left[{g}^u((e^{v,k}+v)(0,t),t)-{g}^u(v(0,t),t)\right] \nonumber\\
& \hspace{1.7cm}  + \left[{g}^u(v(0,t),t)-\bar{g}^u(v(0,t),t)\right], \label{e^u BC 2} \\
&e^{v,k}(1,t) =  \left[ {U}(t) - \bar{U}\left(U(t)\right)\right] - \bar{g}^v({u}(1,t),t), \label{e^v BC 2}\\
&e^{w,k}(x,t_k)=W_k(x)-w(x,t_k). \label{e^w IC 2}
\end{align}
The set $\tilde{\D}_k\cap \bar{\D}_k$ is a determinate set where both the real and the predicted trajectories,  $w$ and $\tilde{w}^{k}$, are independent of their inputs. Therefore,  the error $e^{w,k}$ on $\tilde{\D}_k \cap \bar{\D}_k$ is independent of boundary condition (\ref{e^v BC 2}). 

With $\kappa_1$ as defined in (\ref{kappa_1}), $\|w(x,t)\|\leq \kappa_1 \|w(\cdot,t_k)\|_{\infty}$ for all $(x,t) \in \bar{\D}_k$ (compare with Theorem \ref{thm:determinate set 1} for $k=0$). 
The terms on the right-hand side of (\ref{e^w 2})-(\ref{e^u BC 2}) can be bounded as follows: 
$\|{\Lambda}(x,e^w+w)-{\Lambda}(x,w)\|_{\infty} \leq \|\Lambda_w\|_{\infty}\|e^w\|_{\infty}$, $\|{\Lambda}(x,w)-\bar{\Lambda}(x,{w})\|_{\infty} \leq \epsilon_{\Lambda} \, \kappa_1\, \|w(\cdot,t_k)\|_{\infty}$, $
\|w_x\|_{\infty} \leq l_{ \Lambda^{-1}} (\|w_t\|_{\infty} + l_F\,\|w\|_{\infty} )$,
and similarly for the terms involving $F$.  Consequently, there exist some constants $\gamma_6$ and $\gamma_7$ such that 
\begin{align}
\left|\frac{d}{ds}e^u(\tilde{\xi}^u_k(x,t;s),s)\right| \leq& \gamma_6 \|e^w(\cdot,s) \|_{\infty} \nonumber\\
&  + \gamma_7 (\epsilon_{\Lambda}+\epsilon_F) \|w(\cdot,t_k)\|_{\infty}, \label{hlp 1}\\
\left|\frac{d}{ds}e^v(\tilde{\xi}^v_k(x,t;s),s)\right| \leq& \gamma_6 \|e^w(\cdot,s) \|_{\infty} \nonumber\\
&  + \gamma_7(\epsilon_{\Lambda}+\epsilon_F) \|w(\cdot,t_k)\|_{\infty} \label{hlp 2},
\end{align}
where $\tilde{\xi}^u_k$ and $\tilde{\xi}^v_k$ are defined as in (\ref{xi^u})-(\ref{xi^v}) but with  $\tilde{w}^k$.
The boundary reflection term satisfies $|e^u(0,t) |   \leq |\partial_v g^u| |e^v(0,s) | + \epsilon_{g^u} \kappa_1\|w(\cdot,t_k)\|_{\infty}$,
and (\ref{e^w IC 2}) is bounded by (\ref{error bound w}), i.e., $\|e^{w,k}(x,t_k)\| \leq \epsilon_w \|w(\cdot,t_k)\|_{\infty}$.
%
Applying the comparison result that
$|\dot{\alpha}(t)| \leq a|\alpha(t)| + b$ with $a\geq 0$, $b\geq 0$ and $t\geq 0$, implies
$|\alpha(t)|\leq e^{at}|\alpha(0)|+\frac{b}{a}\left(e^{at}-1\right)\leq e^{at}\left(|\alpha(0)|+\frac{b}{a}\right)$, 
to (\ref{hlp 1})-(\ref{hlp 2}) (with $\alpha(t)$ playing the role of $\|e^w(\cdot,t)\|_{\infty}$, $a=\gamma_6$ and $b=\gamma_7(\epsilon_{\Lambda}+\epsilon_F)\|w(\cdot,t_k)\|_{\infty}$), and keeping in mind that the characteristic lines can have at most one boundary reflection on each determinate set, we obtain the preliminary bound
\begin{equation}\begin{aligned}
\|e^{w,k}(x,t)\| \leq & \gamma_8\,(\epsilon_w+\epsilon_{\Lambda}+\epsilon_F+\epsilon_{g^u})\, \|w(\cdot,t_k)\|_{\infty}\\ 
& \times \exp(\gamma_9\, (t-t_k))  \label{bound preliminary}
\end{aligned}\end{equation}
for some constants $\gamma_8,\,\gamma_9$, on the shared domain of determinacy $(x,t)\in \tilde{\D}_k\cap \bar{\D}_k$.

Next, we investigate the difference between $\tilde{\tau}^v_k(t_k;\cdot)$  and $\bar{\tau}^v(t_k;\cdot)$. Assume without loss of generality that $\tilde{\tau}^v_k(t_k;x)>\bar{\tau}^v(t_k;x)$ for all $x$; if this is not satisfied, the following arguments can be applied by, starting at $x=1$ and moving in the negative $x$-direction, reversing the role of $\tilde{\tau}_k^v$ and $\bar{\tau}^v$ every time the characteristic lines intersect. Omitting the arguments $(t_k;\xi)$ of $\tilde{\tau}^v_k$ in the integrand, we have for all $x\in[0,1]$ 
\begin{equation}\begin{aligned}
&|\tilde{\tau}_k^v(t_k;x)-\bar{\tau}^v(t_k;x)| =\\
& \left|\int_x^1 \frac{1}{\lambda^v(\xi,\tilde{w}^{k}(\xi,\tilde{\tau}_k^v))} - \frac{1}{\lambda^v(\xi,\tilde{w}^{k}(\xi,\bar{\tau}^v))} d\xi \right.\\
& + \int_x^1 \frac{1}{\lambda^v(\xi,\tilde{w}^{k}(\xi,\bar{\tau}^v))} - \frac{1}{\bar{\lambda}^v(\xi,\tilde{w}^{k}(\xi,\bar{\tau}^v))} d\xi \\
&\left. + \int_x^1 \frac{1}{\bar{\lambda}^v(\xi,\tilde{w}^{k}(\xi,\bar{\tau}^v))} - \frac{1}{\bar{\lambda}^v(\xi,w(\xi,\bar{\tau}^v))} d\xi\right|\\
&\leq l_{\Lambda^{-1}}^2 l_{\Lambda}\left(\|\tilde{w}^{k}_t\|_{\infty}\, \|\tilde{\tau}_k^v(t_k;\cdot)-\bar{\tau}^v(t_k;\cdot)\|_{\infty} \right.\\
& \qquad +  \epsilon_{\Lambda}  \kappa_1\|w(\cdot,t_k)\|_{\infty} +  \|e^w(\cdot,\bar{\tau}^v(t_k;\cdot))\|_{\infty} ).
\end{aligned}\end{equation}
Since $l_{\Lambda^{-1}}^2 l_{\Lambda} \|\tilde{w}_t\|_{\infty} \leq l_{\Lambda^{-1}}^2 l_{\Lambda} \kappa_2\bar{\delta}\leq \frac{1}{2} <1$, we can pull the  term with $\|\tilde{\tau}_k^v(t_k;\cdot)-\bar{\tau}^v(t_k;\cdot)\|_{\infty}$ on the right-hand side onto the left-hand side to obtain a bound on $\tilde{\tau}_k^v-\bar{\tau}^v$ depending only on  $e^w(\cdot,\bar{\tau}^v(t_k;\cdot))$ and $\|w(\cdot,t_k)\|_{\infty}$. This gives
\begin{equation}\begin{aligned}
&\|\tilde{w}^{k}(x,\tilde{\tau}_k^v)-w(x,\bar{\tau}^v)\| \\
&\leq \|\tilde{w}^{k}(x,\tilde{\tau}_k^v)-\tilde{w}(x,\bar{\tau}^v)\| + \|e^w(x,\bar{\tau}^v) \|\\
& \leq \|\tilde{w}^{k}_t\|_{\infty}\,|\tilde{\tau}_k^v-\bar{\tau}^v| + \|e^w(x,\bar{\tau}^v) \|_{\infty}\\
&\leq (1+\gamma_{10}) \|e^w(\cdot,\bar{\tau}^v(t_k;\cdot)) \|_{\infty}+\gamma_{10} \epsilon_{\Lambda}\kappa_1\|w(\cdot,t_k)\|_{\infty},\label{bound preliminary 2}
\end{aligned}\end{equation}
where $\gamma_{10}=\kappa_2\bar{\delta}^{\mathrm{max}}\left(1-l_{\Lambda^{-1}}^2 l_{\Lambda}\kappa_2\bar{\delta}^{\mathrm{max}}\right)^{-1}$. Finally, combining (\ref{bound preliminary}) with (\ref{bound preliminary 2}) gives (\ref{bound prediction error}).

\subsection{Proof of Lemma \ref{lemma: input error}} \label{proof appendix 4}
Consider the error between reference and  target trajectories,
\begin{equation}\begin{aligned}
\tilde{e}^{w,k}&(x,t) = \left( \begin{matrix}\tilde{e}^{u,k}(x,t)& \tilde{e}^{v,k}(x,t)\end{matrix} \right)^T \\
 &= \tilde{w}^{*, k}(x,\tilde{\tau}_k^v(t;x))-w^{\mathrm{ref}, k}(x,\bar{\tau}^v(t;x))
\end{aligned}\end{equation}
on the domain $[0,1]\times[t_k,t_{k+1}]$. { Instead of using (\ref{target 1})-(\ref{target 2}), one can show that $\tilde{w}^{*,k}$ also satisfies a PDE-ODE system of the form (see  \cite[Theorem 2]{strecker2017output} for derivation) 
\begin{align}
&\partial_t \tilde{u}^{*,k}= -\mu\,  \partial_x \tilde{u}^{*,k} + \nu\, f^u,&
\partial_x \tilde{v}^{*,k}&= -\frac{f^v}{\lambda^v}, \label{target a1}\\
&\tilde{v}^{*,k}(0,t) = U^{*,k}(\tau^v(t;0)),&
\tilde{u}^{*,k}(0,t) &= g^u(U^{*,k}). \label{target a2}
\end{align}
The terms $\mu$ and $\nu$ contain $\partial_t \tau^v$, and one can show that $|1-\partial_t \tau^v|$ is small due to smallness of $\tilde{w}_t^{*,k}$ and the basic assumptions from Section \ref{sec:assumptions}. Then, the error dynamics  governing $\tilde{e}^{w,k}$ can be formed by subtracting (\ref{target a1})-(\ref{target a2}) from the same set of equations with the actual, uncertain parameters and states. 

}  The initial condition { of $\tilde{e}^{w,k}$} is bounded  by (\ref{bound prediction error}). 
The error in the boundary value at $x=0$ is determined by $|U^{*,k}(\tilde{\tau}_k^v(t;0))-U^{*,\mathrm{ref},k}(\bar{\tau}^v(t;0))|$, $t\in[t_k,t_{k+1}]$, and can therefore  also be bounded by the prediction error $\tilde{e}^{v,k}(0,t_k)$. With this, the preliminary bound
\begin{equation} 
\|\tilde{e}^{w,k}(x,t)\|\leq \tilde{\kappa}_6\,(\epsilon_w+\epsilon_{\Lambda}+\epsilon_F+\epsilon_{g^u})\|w(\cdot,t_k)\|_{\infty} \label{bound input error preliminary}
\end{equation} 
for some $\tilde{\kappa}_6>0$ and all $x\in[0,1],\,t\in [t_k,t_{k+1}]$,   is obtained by following the steps in the proof of Lemma \ref{lemma: prediction error}.

For $t\in[t_k,t_{k+1}]$, the error in the boundary value satisfies
   \begin{equation}\begin{aligned}
   |v(&1,t)-v^{\mathrm{ref}, k}(1,t)|  \\
    \leq&  |v(1,t)-\tilde{v}^{*,k}(1,t)| +    |\tilde{v}^{*,k}(1,t)-v^{\mathrm{ref}, k}(1,t)| \\
     = &|\bar{U}(U(t))-U(t) + \bar{g}^v(u(1,t),t) | +  (\ref{bound input error preliminary})\\
    \leq& \epsilon_U |\tilde{v}^{*, k}(1,t)| + \epsilon_{g^v} |u(1,t)| +  (\ref{bound input error preliminary}), \label{hlp7}
 \end{aligned}  \end{equation}
 where ``(\ref{bound input error preliminary})'' is an abbreviation for the terms on the right-hand side of (\ref{bound input error preliminary}). The first two terms on the right-hand side can be bounded as follows. Like before, the norm of $\tilde{v}$ and $u$ can be bounded based on the norm of the initial and boundary values and the worst-case exponential grow along its characteristic lines: For $\tilde{v}^{*,k}(1,t)$ and $t\in[t_k,t_{k+1}]$,  there exists a  $\tilde{\kappa}_7$ such that
 \begin{equation}
 |\tilde{v}^{*,k}(1,t)| \leq \tilde{\kappa}_7 \Big(  \|\tilde{w}^{*,k}(\cdot,\tilde{\tau}^v_{k}(t_k;\cdot))\|_{\infty}  + \sup_{s\in[\tau_k,\tau_{k+1}]} |U^{*,k}(s)|  \Big) \label{hlp6}
\end{equation}
 where $\|\tilde{w}^{*,k}(\cdot,\tilde{\tau}^v_{k}(t_k;\cdot))\|_{\infty}  \leq \kappa_1 \|\tilde{w}^{k}(\cdot,t_k)\|_{\infty}$, and by design, $\sup_{s\in[\tau_k,\tau_{k+1}]} |U^{*,k}(s)| \leq |\tilde{v}^{k}(0,\tau_k)|$.

 For $u(1,t)$, $t\in[t_k,t_{k+1}]$, we similarly have
  \begin{equation}\begin{aligned}
 |u(1,t)| \leq &\tilde{\kappa}_8 \Big(  \|w(\cdot,t_k)\|_{\infty} + \sup_{s\in[t_k,t_{k+1}]} |v(1,s)|  \Big).
 \end{aligned}\end{equation}
 While it is not possible to find an a priori bound on $ |v(1,\cdot)|$, it can be bounded via
  \begin{equation}
 |v(1,t)| \leq  \, |v^{\mathrm{ref}, k}(1,t)|  + |v(1,t)-v^{\mathrm{ref}, k}(1,t)|
 \end{equation}
 and
\begin{equation}
 |v^{\mathrm{ref}, k}(1,t)| \leq \,\tilde{\kappa}_9 \Big(\|w(\cdot,t_k)\|_{\infty}  +\sup_{s\in[\bar{\tau}_k,\bar{\tau}_{k+1}]} |U^{*,ref}(t_k;s)|   \Big),\label{hlp8}
\end{equation} 
 where $|U^{*,ref}(t_k;s)|$ can again be bounded via $\|w(\cdot,t_k)\|_{\infty}$. In summary, inserting (\ref{hlp6})-(\ref{hlp8}) into (\ref{hlp7}) gives an estimate of the form
  \begin{equation}\begin{aligned}
   |v(1,t)-v^{\mathrm{ref}, k}(1,t)|  \leq  & \, \epsilon_{g^v} \tilde{\kappa}_{10} |v(1,t)-v^{\mathrm{ref}, k}(1,t)| \\
   &~ + \tilde{\kappa}_{11} \,\epsilon_{\mathrm{sum}}\, \|w(\cdot,t_k)\|_{\infty} \label{hlp9}.
 \end{aligned}  \end{equation}
Finally, (\ref{bound input error})   is  obtained  since $\epsilon_{g^v} \tilde{\kappa}_{10} < 1$.


%

%
%
%
%
%

\ifCLASSOPTIONcaptionsoff
  \newpage
\fi



%
\bibliographystyle{IEEEtran}
\bibliography{IEEEabrv,references} 
%

\end{document}